\documentclass[12pt,reqno]{amsart}
\usepackage{appendix}
\usepackage{amsmath, amssymb,upgreek,marvosym,amsthm,caption,subcaption,amstext,array,euler,times}
\usepackage{kpfonts}
\usepackage[margin=1.15in]{geometry}
\numberwithin{equation}{section}
\usepackage{algorithm}
\usepackage{algorithmic}
\makeatletter
\def\listofalgorithms{\@starttoc{loa}\listalgorithmname}
\def\l@algorithm{\@tocline{0}{3pt plus2pt}{0pt}{1.9em}{}}
\renewcommand{\ALG@name}{Algorithm}
\renewcommand{\listalgorithmname}{List of \ALG@name s}
\numberwithin{algorithm}{section}
\theoremstyle{definition}

\theoremstyle{remark}
\newtheorem{rem}[algorithm]{Remark}
\theoremstyle{theorem}
\newtheorem{thm}[algorithm]{Theorem}
\theoremstyle{proposition}
\newtheorem{prop}[algorithm]{Proposition}
\theoremstyle{example}
\newtheorem{example}[algorithm]{Example}
\theoremstyle{corollary}

\makeatother

\usepackage{scalerel}

\usepackage{lipsum}
\usepackage{hyperref}
\hypersetup{
    colorlinks=true,
    linkcolor=blue,
    filecolor=magenta,
    urlcolor=cyan,
}

\allowdisplaybreaks

\begin{document}
\title[Generalized Hessians for optimal value functions]{Estimates of generalized Hessians for optimal value functions in mathematical programming}
\author[]{Alain B. Zemkoho\\[6ex] 
S\MakeLowercase{chool of} M\MakeLowercase{athematics}, U\MakeLowercase{niversity of} S\MakeLowercase{outhampton}, UK\\
\MakeLowercase{\textsf{a.b.zemkoho@soton.ac.uk}}\\\\
\emph{T}\MakeLowercase{\emph{his work is dedicated to}} \emph{P}\MakeLowercase{\emph{rofessor}} \emph{S}\MakeLowercase{\emph{tephan}} \emph{D}\MakeLowercase{\emph{empe on the occasion of his 65th birthday}}}

\thanks{This work is supported by an EPSRC grant {EP/V049038/1} and the Alan Turing Institute under the EPSRC grant {EP/N510129/1}.}%

\date{\today}%
\begin{abstract}
%
We consider the optimal value function of a parametric optimization problem. A large number of publications have been dedicated to the study of continuity and differentiability properties of the function. However, the differentiability aspect of works in the current literature has mostly been limited to first order analysis, with focus on estimates of its directional derivatives and subdifferentials, given that the function is typically nonsmooth. With the progress made in the last two to three decades in major subfields of optimization such as robust, minmax, semi-infinite and bilevel optimization, and their connection to the optimal value function, there is a need for a second order analysis of the generalized differentiability properties of this function. This could enable the development of robust solution algorithms, such as the Newton method. The main goal of this paper is to provide  estimates of the generalized Hessian for the optimal value function. Our results are based on two handy tools from parametric optimization, namely the optimal solution and Lagrange multiplier mappings, for which completely detailed estimates of their generalized derivatives are either well-known or can easily be obtained.
\end{abstract}
\maketitle
\section{Introduction}
\subsection{Aim of the  work} Considering the functions $f$ and $g$ defined from $\mathbb{R}^{n+m}$ to $\mathbb{R}$ and $\mathbb{R}^p$, respectively, we are interested in the parametric optimization problem
\begin{equation}\label{ParametricOptimization}
   \underset{y}\min~\left\{f(x,y)|\; g(x,y)\leq 0 \right\}.
\end{equation}
We only consider inequality constraints, in order to focus our attention on the main ideas. Note however that all the results in this paper remain valid, with the corresponding adjustments, if we include equality constraints to problem \eqref{ParametricOptimization}.
Our focus will be on the optimal value function
\begin{equation}\label{OptValueFn}
   \varphi(x):= \underset{y}\min~\left\{f(x,y)|\; g(x,y)\leq 0 \right\},
\end{equation}
related to the parametric optimization problem \eqref{ParametricOptimization}. Another object closely related to problem \eqref{ParametricOptimization} is the optimal solution set-valued mapping $S : \mathbb{R}^n \rightrightarrows \mathbb{R}^m$ defined by
\begin{equation}\label{S(x)}
S(x):= \arg\underset{y}\min~\left\{f(x,y)|\; g(x,y)\leq 0 \right\}.
\end{equation}
We assume throughout the paper that $S(x)\neq \emptyset$ for all $x\in \mathbb{R}^n$. As a consequence, $\varphi$ will be finite-valued  at all $x\in \mathbb{R}^n$. For most results obtained in this paper, we can easily accommodate the case where $\varphi$ \eqref{OptValueFn} is an extended real-valued function. But to concentrate on the main points, we leave this specific case for future analysis.

The function $\varphi$ \eqref{OptValueFn} has played a major role in the development and understanding of the structure of the underlying parametric optimization problem \eqref{ParametricOptimization}, and has been substantially analyzed in the literature. Initial work on stability/sensitivity analysis of optimization problems is almost as old as the field of optimization itself, given that early works on linear programming and the simplex method already provided interesting insights on the behavior of optimal values under perturbations; see, example, the 12th Chapter of the 1963 book by Dantzig \cite{Dantzig1963}.

The study of continuity and differential properties of $\varphi$, in the context of nonlinear optimization, which is our main focus here, grew dramatically following works by Fiacco \cite{FiaccoBook1983}, Gauvin and Dubeau \cite{GauvinDubeau1982}, amongst many others. Recent publications on the topic include the papers \cite{MordukhovichNamPhanVarAnalMargBlP,MordukhovichNamYenSubgradients2009}, where the tools by Mordukhovich are used to provide different types of subdifferential estimates for $\varphi$.

It should be emphasized that most of the aforementioned works focus on the derivation of continuity properties and the estimation of directional derivatives and subdifferentials. As far as second order differentiation properties for $\varphi$ are concerned, a few publications have been devoted to the estimation of second order-type directional derivatives; see, e.g., \cite{Shapiro1985second, Shapiro1988perturbation}. We are however not aware of any work pursuing generalized Hessian evaluations for $\varphi$, in the sense considered in this paper; cf. \eqref{Definition of Hessian phi}. Fiacco \cite[Chapter 3]{FiaccoBook1983} provides Hessian formulas for $\varphi$ in the case where $S$ \eqref{S(x)} is single-valued and continuously differentiable. This assumption is very restrictive and cannot hold for most applications. However, inducing second order sufficient conditions for approximating
problems can give rise to efficient smoothing methods for many problems involving nonsmooth functions  \cite{borges2020regularized, borges2021decomposition}.

 We consider various scenarios in this paper, including the case where $S$ \eqref{S(x)} is single-valued and continuously differentiable, leading to  results that coincide with those by Fiacco just mentioned above; see details in the next subsection and in Section \ref{Generalized Hessian estimates}. In \cite{ButikoferKlatteKummer2010Second}, generalized Taylor expansion and other second order generalized differentiation tools are applied to a function closely related to $\varphi$; but their result are completely different from the ones obtained here.

%
%

 It is also important to recall that the optimal value function naturally appears either in the constraints or objective functions of many mainstream optimization problems, including \emph{robust} \cite{ben2009robust}, \emph{minmax} \cite{DanskinBook1967}, \emph{generalized semi-infinite} \cite{RuckmannShapiro2001}, and \emph{bilevel optimization} \cite{DempeFoundations, Zemkoho2016Solving}. Considering the recent developments in the construction of Newton-type methods using the Mordukhovich generalized second order differentiation tools (see \cite{khanh2020generalized, mordukhovich2020generalized, mordukhovich2020globally}), the results of this paper could play an important role in solving the aforementioned problems.

Our main goal in this paper is to provide estimates of the Generalized Hessian of the optimal value function $\varphi$, in the sense of Mordukhovich. Note that for $\bar x\in \mathbb{R}^n$ and $\underline{x}\in \partial\varphi(\bar x)$, the generalized Hessian of $\varphi$ in the sense of Mordukhovich is defined by
\begin{equation}\label{Definition of Hessian phi}
\partial^2 \varphi(\bar x| \underline{x})(\underline{x}^*) := D^*(\partial\varphi)(\bar x| \underline{x})(\underline{x}^*)
\end{equation}
for all $\underline{x}^*\in \mathbb{R}^n$. Here, $\partial\varphi$ denotes the subdifferential of $\varphi$ in any standard sense (convex analysis, Clarke or Mordukhovich) and $D^*(\partial\varphi)$ stands for the coderivative of Mordukhovich \cite{MordukhovichBook2006}. In the next subsection, we provide a general flavour of how the estimates developed in this paper look like. Further details on the mathematical tools and the proofs of the main results are provided in Sections \ref{Mathematical tools needed}--\ref{Generalized Hessian estimates in the absence of convexity}. 

%

\subsection{Summary of the main results--outline of the paper}\label{Main results and outline of paper} We assume throughout the paper that the functions $f$ and $g$, cf. \eqref{ParametricOptimization}, are twice continuously differentiable.
If $S$ \eqref{S(x)} reduces to a single-valued function around a given point, the notation symbol will be adjusted to $s$. We use the set-valued map $\Lambda :\mathbb{R}^{n+m} \rightrightarrows \mathbb{R}^p$ to collect all the \emph{Lagrange multipliers} of problem \eqref{ParametricOptimization}; i.e.,
\begin{equation}\label{Lambdadefined}
\Lambda(x,y) := \left\{u\in \mathbb{R}^p|\, \nabla_yL(x,y,u)=0, \, u\geq 0, \, g(x,y)\leq 0, \, u^\top g(x,y)=0\right\},
\end{equation}
 where $L(x,y,u):=f(x,y) + u^\top g(x,y)$ denotes the \emph{Lagrangian function} of  problem \eqref{ParametricOptimization}.
Similarly, $\lambda$ will be used to represent $\Lambda$ when it is single-valued.

If the constraint function $g$ is independent from the parameter $x$ and the corresponding optimal solution set-valued map $S$ is single-valued and locally Lipschitz continuous around $\bar x$, then we show in Subsections \ref{Case where the feasible set is unperturbed}--\ref{Single-valued solution and multipliers maps} that under further appropriate conditions, the generalized Hessian of $\varphi$ \eqref{OptValueFn} can be obtained as
\begin{equation}\label{man-no-lap1}
   \partial^2 \varphi(\bar x)(x^*) \subseteq \nabla^2_{xx}f(\bar x, \bar y) x^* + \bar\partial s(\bar x)^\top \nabla^2_{xy}f(\bar x, \bar y) x^*,
\end{equation}
where $\bar y= s(\bar x)$ and $\bar\partial s(\bar x)^\top$ stands for the set of transposed generalized Jacobians of $s$ in the sense of Clarke, cf. \eqref{Clarke Subifferential}. If we further impose the continuous differentiability on $s$ at the point $\bar x$, then we can obtain the equality
\begin{equation}\label{man-no-lap2}
   \partial^2 \varphi(\bar x)(x^*) = \nabla^2_{xx}f(\bar x, \bar y) x^* + \nabla s(\bar x)^\top \nabla^2_{xy}f(\bar x, \bar y) x^*,
\end{equation}
which coincides with the result in \cite[Corollary 3.4.2(c)]{FiaccoBook1983}.
 Conditions ensuring that $S$ is single-valued and locally Lipschitz continuous or continuously differentiable are recalled in Section \ref{On the subdifferential}. The results above can further be generalized to the case where the optimal solution mapping $S$ \eqref{S(x)} is multi-valued. Precisely, under appropriate conditions, we also obtain in Subsections \ref{Case where the feasible set is unperturbed}--\ref{Single-valued solution and multipliers maps} that
\begin{equation}\label{man-no-lap3}
   \partial^2 \varphi(\bar x| \underline{x})(\underline{x}^*) \subseteq \underset{y\in S(\bar x): \;\, \underline{x}=\nabla_x f(\bar x, y)}\bigcup\left[\nabla^2_{xx}f(\bar x, y)\underline{x}^*  + D^*S(\bar x| y)\left(\nabla^2_{xy}f(\bar x, y)\underline{x}^*\right)
  \right].
\end{equation}
If we drop the assumption that the feasible set of the parametric optimization problem \eqref{ParametricOptimization} is unperturbed, we can obtain the following estimate for the generalized Hessian of $\varphi$, provided that $S$ \eqref{S(x)} and $\Lambda$ \eqref{Lambdadefined} are both single-valued and Lipschitz continuous around $\bar x$ and $(\bar x, \bar y)$, respectively,  with $\bar y=s(\bar x)$ and $\bar u = \lambda(\bar x, \bar y)$:
\begin{equation}\label{nemama}
   \begin{array}{l}
\partial^2 \varphi(\bar x| \underline{x})(\underline{x}^*) \subseteq \nabla^2_{xx}L(\bar x, \bar y, \bar u)\underline{x}^* \\[1ex]
\qquad\qquad  + \quad \underset{(\zeta^*_x, \zeta_y)\in \partial \langle \nabla_x g(\bar x, \bar y)\underline{x}^*, \;\,\lambda \rangle (\bar x, \bar y)} \bigcup \left[ \zeta^*_x +  \partial \langle \zeta^*_y + \nabla^2_{xy}L(\bar x, \bar y, \bar u)\underline{x}^*, \;\,s \rangle (\bar x)\right].
   \end{array}
\end{equation}
Here, the symbol $\partial$ refers to the subdifferential, in the sense of Mordukhovich, cf. \eqref{Basic Subdifferential}. Obviously, as in the case of \eqref{man-no-lap2}, supposing that $\lambda$ and $s$ are both single-valued and differentiable functions, we will have
\begin{equation}\label{mamba}
   \begin{array}{l}
\partial^2 \varphi(\bar x| \underline{x})(\underline{x}^*) = \nabla^2_{xx}L(\bar x, \bar y, \bar u)\underline{x}^* + \nabla s(\bar x)^\top\nabla^2_{xy}L(\bar x, \bar y, \bar u)\underline{x}^* \\
\qquad\qquad\qquad \quad+  \left[\nabla_x \lambda(\bar x, \bar y)^\top +  \nabla s(\bar x)^\top \nabla_y \lambda(\bar x, \bar y)^\top \right] \nabla_x g(\bar x, \bar y)\underline{x}^*.
   \end{array}
\end{equation}
Observing that $\nabla_x \lambda(\bar x, \bar y)^\top +  \nabla s(\bar x)^\top \nabla_y \lambda(\bar x, \bar y)^\top $ corresponds to the Jacobian of the function $x \mapsto \lambda (x, s(x))$, it is clear that \eqref{mamba} coincides with the formula obtained in \cite[Corollary 3.4.1(c)]{FiaccoBook1983}. More details on how the estimate in  \eqref{nemama} is obtained are given in Subsection \ref{Single-valued solution and multipliers maps}. For their generalizations to the case where $S$ \eqref{S(x)} is single-valued and $\Lambda$ \eqref{Lambdadefined} is set-valued, see Subsection \ref{Single-valued optimal solution map}. As for more general results, in the absence of convexity, see Subsection \ref{Set-valued optimal solution map}.  Before moving to that, we first provide some background results, which are useful in their own right, in particular, in the understanding of the subdifferential of $\varphi$ \eqref{OptValueFn} and the \emph{Lipschitz-likeness}, and \emph{generalized differentiability} properties of the related mappings $S$ \eqref{S(x)} and $\Lambda$ \eqref{Lambdadefined}; cf. Section \ref{On the subdifferential}.

 Some final comments on limitations and potential applications of the results from this paper, as well as topics for future research, are discussed in Section \ref{Final discussion and future work}.

\section{Notation and mathematical tools needed}\label{Mathematical tools needed}
We start this section with some notation and basic concepts used throughout the paper. We will use $v_j$, $j=1, \ldots, n$, to denote the $j$th component  of a vector $v\in \mathbb{R}^n$, while $v^i\in \mathbb{R}^n$, $i=1, \ldots, m$, will represent the $i$th vector component  of a vector of vectors $v\in \prod^m_{i=1}\mathbb{R}^n$. Furthermore, for $v\in \mathbb{R}^n$ and $I \subseteq \{1, \ldots, n\}$, writing $v_I$ will refer to all components $v_i$ of $v$ for which $i\in I$. To avoid confusions at some points, we will use $\{0_n\}$ or $0_n$ for a $n$-dimensional zero vector. When a distinction is also necessary, we will use $I_n$ for the $n\times n$ identity matrix. We use the notation $\Psi : \mathbb{R}^n \rightrightarrows \mathbb{R}^m$ for a set-valued mapping and the \emph{lower case} form $\psi : \mathbb{R}^n \rightarrow \mathbb{R}^m$ to symbolize a single-valued mapping. Most notably, as already mentioned in the previous section, such transitions from \emph{upper} to \emph{lower case} will be used for the optimal solution and Lagrange multipliers set-valued mappings $S$ \eqref{S(x)}  and $\Lambda$ \eqref{Lambdadefined}, in the forms $s$ and $\lambda$, respectively. Additionally, for a set-valued mapping $\Psi : \mathbb{R}^n \rightrightarrows \mathbb{R}^m$, it will be said to be closed if its graph denoted by $\mbox{gph }\Psi:=\{(x,y)\in \mathbb{R}^{n}\times \mathbb{R}^{m}|\; y\in \Psi(x)\}$ is a closed subset of $\mathbb{R}^{n}\times \mathbb{R}^{m}$. Also recall that for a set $C\subseteq \mathbb{R}^n$, $\mbox{co }C$ will be used to denote the convex hull of $C$.

For a closed subset $C$ of $\mathbb{R}^n$, the \emph{Mordukhovich} (also known as basic or limiting) \emph{normal cone} to $C$ at one of its points $\bar x$ is the set (see, e.g., \cite{MordukhovichBook2006})
\begin{equation}\label{basic normal cone}
 N_C(\bar x):= \left\{v\in \mathbb{R}^n|\; \exists v_k \rightarrow v, \, x_k \rightarrow \bar x \,(x_k\in C): \, v_k\in \widehat{N}_C(x_k)\right\},
\end{equation}
where $\widehat{N}_C$ denotes the dual of the contingent/Boulingand tangent cone to $C$:
$$
 \widehat{N}_C(\bar x):=\left\{v\in \mathbb{R}^n|\; \langle v, u-\bar x\rangle \leq o(\|u-\bar x\|)\;\; \forall u\in C\right\}.
$$
We have the following well-known result, which can be found in \cite{MordukhovichGeneralizedDifferential1994,RockafellarWetsBook1998}.
\begin{thm}\label{normal cone estimate}{\em
Let $C:=\psi^{-1}(\Xi)$, where $\Xi\subseteq \mathbb{R}^m$ is a closed set and $\psi :\mathbb{R}^n \rightarrow \mathbb{R}^m$ a Lipschitz continuous function around $\bar x$, then we have
\begin{equation}\label{normal cone to operator constraint}
   N_C(\bar x) \subseteq     \bigcup\left\{\partial \langle v, \psi \rangle (\bar x)\big|\; v\in N_\Xi(\psi(\bar x)) \right\},
\end{equation}
provided the following basic-type qualification condition is satisfied at $\bar x$:
\begin{equation}\label{Basic CQ2}
    \left[0\in  \partial \langle v, \psi \rangle (\bar{x}), \; v\in   N_{\Xi}(\psi(\bar{x}))\right] \Longrightarrow v=0.
\end{equation}
}\end{thm}
Equality holds in \eqref{normal cone to operator constraint}, provided that the set $\Xi$ is normally regular at $\psi(\bar x)$, i.e., $N_{\Xi}(\psi(\bar{x}))=\widehat{N}_{\Xi}(\psi(\bar{x}))$. This is obviously the case if $\Xi$ is a convex set.

In \eqref{normal cone to operator constraint} and \eqref{Basic CQ2}, the term $\partial \langle v, \psi \rangle (\bar x)$ refers to the Mordukhovich subdifferential of the function $x \mapsto \sum^m_{i=1} v_i \psi_i(x)$ at $\bar x$.
If $\psi :\mathbb{R}^n \rightarrow \mathbb{R}$, then the \emph{Mordukhovich} (also known as basic or limiting) \emph{subdifferential} of $\psi$ at $\bar x$ can be defined by
\begin{equation}\label{Basic Subdifferential}
\partial \psi (\bar x):=\left\{\xi\in \mathbb{R}^n|\, (\xi, -1)\in N_{\text{epi}\psi}(\bar x, \psi(\bar x))\right\},
\end{equation}
where $\text{epi}\psi$ stands for the epigraph of $\psi$. If $\psi$ is Lipschitz continuous around $\bar x$, then we can also define the \emph{Clarke} (or convexified) \emph{subdifferential} of $\psi$ at $\bar x$:
\begin{equation}\label{Clarke Subifferential}
\Bar{\partial} \psi (\bar{x}) := \text{co}\, \partial \psi(\bar{x}).
\end{equation}
In the case where $\psi$ is convex,  $\partial \psi (\bar{x})$ and $\Bar{\partial} \psi (\bar{x})$ coincide with the subdifferential in the sense of convex analysis.

Using the above concept of basic normal cone, we now  introduce the notion of {\em coderivative} for a given set-valued map $\Psi :\mathbb{R}^n \rightrightarrows \mathbb{R}^m$,  at some point $(\bar{x}, \bar{y})\in \text{gph}\, \Psi$, which corresponds to a homogeneous mapping $D^*\Psi(\bar{x}| \bar{y}): \mathbb{R}^m \rightrightarrows \mathbb{R}^n$, defined by
\begin{equation}\label{cod-definition}
   D^*\Psi(\bar{x}|\bar{y})(y^*):= \left\{x^*\in \mathbb{R}^n| (x^*, -y^*)\in N_{\text{gph}\, \Psi}(\bar{x}, \bar{y})\right\},
\end{equation}
for all $y^*\in \mathbb{R}^m$. Here, 
$N_{\text{gph}\, \Psi}$ represents the basic normal cone \eqref{basic normal cone} to $\text{gph}\,\Psi$.
The following chain rule from \cite[Theorem 5.1]{MordukhovichGeneralizedDifferential1994} will be pivotal in this work.
\begin{thm}\label{chain rule}{\em
Let the set-valued mappings $\Phi : \mathbb{R}^n\rightrightarrows\mathbb{R}^m$ and $F : \mathbb{R}^m\rightrightarrows\mathbb{R}^q$ have closed graph. Furthermore, let $\bar z\in (F \circ \Phi)(\bar x)$ and assume that the set-valued map
\begin{equation}\label{M(xz)}
    M(x,z):=\Phi(x)\cap F^{-1}(z) = \left\{y\in \Phi(x)|\; z\in F(y)  \right\}
\end{equation}
is locally bounded around $(\bar x, \bar z)$ and the qualification condition
\begin{equation}\label{papasal}
D^*F(y|\bar z)(0)\cap \mbox{Ker }D^*\Phi(\bar x|y) = \{0\} \;\; \forall y\in \Phi(\bar x) \cap F^{-1}(\bar z)
\end{equation}
is fulfilled. Then for all $z^*\in \mathbb{R}^q$, we have
$$
D^*(F\circ\Phi)(\bar x| \bar z)(z^*) \subseteq \underset{y\in \Phi(\bar x)\cap F^{-1}(\bar z)}\bigcup D^*\Phi(\bar x| y)(z^*)\circ D^*F(y| \bar z)(z^*).
$$
}\end{thm}
The following result from \cite[Proposition 3.3]{DempeZemkohoKKT-SIAM-paper2}, providing a coderivative estimate for a Cartesian product of finitely many set-valued mappings, will also be useful in the development of the main results of this paper.
\begin{thm}\label{pp1}{\em
Consider the set-valued mappings $\Psi_i : \mathbb{R}^n \rightrightarrows \mathbb{R}^q$ for $i=1, \ldots, p$, and define a Cartesian product mapping $\Psi :\mathbb{R}^n \rightrightarrows \mathbb{R}^{q\times p}$ by
\begin{equation*}\label{Cart-Prod}
   \Psi(x):= \prod^p_{i=1} \Psi_i(x) = \Psi_1(x)\times \ldots \times \Psi_p(x).
\end{equation*}
Assume that $\text{gph}\,\Psi_i$, $i=1, \ldots, p$, is closed and the qualification condition
\begin{equation}\label{QC-Cod}
   \left[\sum^p_{i=1}v^i=0, \; v^i\in D^*\Psi_i(\bar x| \bar{y}^i)(0), \; i=1, \ldots, p\right] \Longrightarrow v^1=\ldots = v^p=0
\end{equation}
 is satisfied at $(\bar x, \bar y)$ with $\bar{y}:=(\bar{y}^i)^p_{i=1}\in \Psi(\bar{x})$. Then, for any $y^* :=(y^{*i})^p_{i=1}\in \prod^p_{i=1}\mathbb{R}^q$,
\begin{equation}\label{Cod-Prod}
    D^*\Psi(\bar x| \bar y)(y^*) \subseteq \sum^p_{i=1} D^*\Psi_i(\bar x| \bar{y}^i)(y^{*i}).
\end{equation}
Equality holds in \eqref{Cod-Prod}, if $\text{gph}\, \Psi_i$ is normally regular at $(\bar x, \bar{y}^i)$, for $i=1, \ldots, p$.
}\end{thm}

Next, we provide an estimate of the coderivative of a set-valued mapping defined by the convex hull of another set-valued mapping, which is not necessarily convex-valued. To proceed, consider $\Psi : \mathbb{R}^n \rightrightarrows \mathbb{R}^m$ and define $\Phi : \mathbb{R}^n \rightrightarrows \mathbb{R}^m$ by
\begin{equation}\label{coMapping}
  \Phi(x):= \mbox{co }\Psi(x).
\end{equation}
 $\Psi$ is assumed to be nonconvex-valued at some points of $\mathbb{R}^n$, as ``co'' can obviously be dropped at points where the map is convex. An upper estimate of the coderivative of $\Phi$ in terms of the coderivative of $\Psi$ can then be obtained as follows. To make the presentation of the result easier, we introduce the set 
\begin{equation}\label{dfort}
    \begin{array}{c}
  \Gamma (\bar x, \bar y):= \left\{\left.(a,b)\in \mathbb{R}^{m+1}\times \prod^{m+1}_{s=1}\mathbb{R}^{m}\right|\right. \; a\geq 0, \; \sum^{m+1}_{s=1}a_s=1, \\
  \qquad \qquad \qquad  \qquad \qquad \qquad \quad \left.\sum^{m+1}_{s=1}a_sb^s=\bar y, \; b :=(b^s)^{n+1}_{s=1}\in \prod^{m+1}_{s=1}\Psi(\bar x)\right\}.
\end{array}
\end{equation}
\begin{prop}\label{transmon}{\em
Consider $(\bar x, \bar y)\in \mbox{gph }\Phi$ and suppose that the set-valued mapping $\Psi$  \eqref{coMapping} is closed and locally bounded around $\bar x$. Furthermore, assume that \eqref{QC-Cod}, with $\Psi_s:=\Psi$ for $s=1, \ldots, m+1$, holds for all $(a,b)\in \Gamma(\bar x, \bar y)$. Then for all $y^* \in \mathbb{R}^m$, we have
\begin{equation}\label{Coderivative of ConvexHull}
  D^*\Phi(\bar x| \bar y)(y^*) \subseteq \underset{(a, b)\in \Gamma (\bar x, \bar y)}\bigcup \left[\sum^{m+1}_{s=1} D^*\Psi\left(\bar x| b^s\right)\left(a_s y^*\right)\right].
\end{equation}
}\end{prop}
\begin{proof}Start by recalling that as $\Phi(x) \subseteq \mathbb{R}^m$ for all $x\in \mathbb{R}^n$, it follows from the well-known Theorem of Carath\'{e}odory that $\Phi(x)$ can be rewritten as
$$
\Phi(x) = \left\{\left.\sum^{m+1}_{s=1} \eta_s v^s\right|\; \eta_s\geq 0, \; v^s \in \Psi(x),\; s=1, \ldots, m+1, \sum^{m+1}_{s=1} \eta_s=1   \right\}.
$$
Based on this expression, we can easily check that $\Phi$ can take the form
$$
\begin{array}{l}
  \Phi(x) = \ell \circ Q(x)\\
  \mbox{with } \left\{
  \begin{array}{l}
    \ell(a,b):= \sum^{m+1}_{s=1} a_s b^s,\\
    Q(x) := \Xi \times \prod^{m+1}_{s=1}\Psi(x),\\
    \Xi:=\left\{a\in \mathbb{R}^{m+1}| \; a\geq 0, \; \sum^{m+1}_{s=1} a_s=1 \right\}.
  \end{array}
  \right.
\end{array}
$$
Considering the continuous differentiability of $\ell$, the closedness of the set $\Xi$ and the set-valued mapping $\Psi$, it follows from the chain rule above, cf. Theorem \ref{chain rule}, that for $\bar x\in \mathbb{R}^n$ and $\bar y\in \Phi(\bar x)$, it holds that
\begin{equation}\label{mekre}
D^*\Phi(\bar x| \bar y)(y^*) \subseteq \underset{(a, b)\in Q(\bar x)\cap \ell^{-1}(\bar y)}\bigcup \left[D^*Q(\bar x| a,b)(\nabla \ell(a,b)^\top y^*)  \right]
\end{equation}
for $y^*\in \mathbb{R}^m$, provided that set-valued mapping
$
M(x,y):=\left\{(a,b)\in Q(x)|\; \ell(a,b)=y \right\},
$
counterpart of \eqref{M(xz)} is locally bounded around $(\bar x, \bar y)$.
Obviously, from the definition of this mapping, $M(x,y)\subseteq \Xi \times \prod^{m+1}_{s=1}\Psi(x)$ for all $(x,y)$. Hence, $M$ is locally bounded around $(\bar x, \bar y)$ given that $\Xi$ is a bounded set and $\Psi$ is assumed to be locally bounded around $\bar x$.
Now observe that any $b :=(b^s)^{n+1}_{s=1}\in \prod^{m+1}_{s=1} \mathbb{R}^m$ is a $m\times (m+1)$ matrix. We rearrange it as a $m^2+m$-dimensional  column vector and proceed with the following  notation for the rest of the proof:
$$
b:=\left[b^1_1 \ldots b^1_m\; \ldots \;  b^{m+1}_1 \ldots b^{m+1}_{m}\right]^\top \; \mbox{ and }\; \bar b := \left[\begin{array}{ccc}
     b^1_1 & \ldots  & b^{m+1}_1 \\
     \vdots & \ldots  & \vdots \\
     b^1_m & \ldots  & b^{m+1}_m
   \end{array}\right].
$$
Then simple calculations show that $\nabla \ell(a,b) = \left[\bar b \;\; a_1I_m \;\; \ldots\;\; a_{m+1}I_m \right]$.
Hence,
$$
\nabla \ell(a,b)^\top y^* = \left[(\bar b^\top y^*)^\top, \; (a_1y^*)^\top,\;  \ldots, \; (a_{m+1}y^*)^\top\right]^\top
$$
for $y^*\in \mathbb{R}^m$. Considering this formula, the application of Theorem \ref{pp1} to the set-valued mapping $Q$ at the vector $\nabla \ell(a,b)^\top y^*$ leads to the inclusion
\begin{equation}\label{mamama}
    D^*Q(\bar x| a,b)(\nabla \ell(a,b)^\top y^*) \subseteq \sum^{m+1}_{s=1} D^*\Psi\left(\bar x| b^s\right)\left(a_s y^*\right),
\end{equation}
given that the coderivative of the constant mapping defined by $\Xi$ is $\{0_n\}$ and   condition \eqref{QC-Cod} is assumed to hold for all $(a, b)$  such that $(a, b)\in Q(\bar x)$ and $\ell(a,b)= \bar y$. Then combining \eqref{mamama} with \eqref{mekre}, we have the result.
\end{proof}
\begin{rem}
Note that we also have from Theorem \ref{pp1} that equality holds in \eqref{Coderivative of ConvexHull}, if $\text{gph}\, \Psi$ is normally regular at $(\bar x, b^s)$ for $s=1, \ldots, m+1$. Furthermore, one can easily check that \eqref{Coderivative of ConvexHull} is a natural extension of the coderivative of $\Phi$ \eqref{coMapping} when $\Psi$ is single-valued and differentiable. In fact, in the latter case,
$
D^*\Phi(\bar x| \bar y)(y^*)= \nabla \Psi (\bar x)^\top y^*.
$
\end{rem}
To close this section, we introduce the  Lipschitz-likeness property that will used in the next section. A set-valued mapping $\Psi :\mathbb{R}^n \rightrightarrows \mathbb{R}^m$ is \emph{Lipschitz-like} at $(\bar x, \bar y)\in \mbox{gph }\Psi$ if there are neighborhoods $U$ of $\bar{x}$, $V$ of $\bar{y}$, and a constant $\kappa>0$ such that
$
d(y,\Psi(x))\le\kappa\|x-u\|\;\mbox{ for all }\;x,u\in U\;\mbox{ and }\;y\in\Psi(u)\cap V,
$
where $d$ stands for the usual distance function.  We can add that a weaker Lipschitz property, known as \emph{calmness}, can be obtained if we fix $u$ to $x$;
precisely, $\Psi $ is \emph{calm} at $(\bar x, \bar y)\in \mbox{gph }\Psi$ if there are neighborhoods $U$ of $\bar{x}$, $V$ of $\bar{y}$, and a constant $\kappa>0$ such that
$
d(y,\Psi(x))\le\kappa\|x-\bar x\|\;\mbox{ for all }\;x\in U\;\mbox{ and }\;y\in\Psi(u)\cap V.
$
A closed set-valued mapping $\Psi$ is  Lipschitz-like around $(\bar x, \bar y)$ if and only if the condition
\begin{equation}\label{coderivative criterion}
D^*\Psi(\bar{x}|\bar{y})(0)=\{0\},
\end{equation}
known as the \emph{coderivative/Mordukhovich criterion}, is satisfied at $(\bar x, \bar y)$; cf. \cite[Theorem~5.7]{MordukhovichBook2006} and \cite[Theorem~9.40]{RockafellarWetsBook1998}. Observe for instance that if this criterion holds for $F$ in \eqref{papasal} and $\Psi_i$ ($i=1,\ldots, p$) for \eqref{QC-Cod}, then the corresponding conditions are automatically satisfied. This is therefore the case if these set-valued maps are closed and Lipschitz-like around the corresponding points.

\section{On the subdifferential of the optimal value function}\label{On the subdifferential}
To start this subsection, we recall that the fundamental goal of this paper is to develop \emph{generalized Hessians} (also known as \emph{second order subdifferentials}) of the optimal value function $\varphi$ \eqref{OptValueFn}. Hence, it would be natural to first clarify the expressions of the \emph{subdifferentials} or \emph{first order subdifferentials}, to be precise, of this function. These quantities and further properties have been extensively studied in the literature; see, e.g., \cite{BonnansShapiroBook2000, DanskinBook1967, FiaccoBook1983, GauvinDubeau1982, MordukhovichNamPhanVarAnalMargBlP,MordukhovichNamYenSubgradients2009,ShimizuIshizukaBardBook1997} and references therein. 
Below, we recall the relevant aspects of these properties while adding some crucial aspects based on the \emph{concave-convexity}, that we define below. Before, note that from here on, the feasible set of the parametric problem \eqref{ParametricOptimization} will be defined by the following set-valued mapping:
\begin{equation}\label{K(x)}
    K(x):=\left\{y\in \mathbb{R}^m|\; g(x,y)\leq 0 \right\}.
\end{equation}
 A  function $\psi$ defined from $\mathbb{R}^{n+m}$  to $\mathbb{R}$ by $(x,y) \mapsto \psi(x,y)$ will be said to be \emph{concave-convex} if the function $\psi(., y)$ is concave for all $y\in \mathbb{R}^m$, while the function $\psi(x, .)$ is convex for all $x\in \mathbb{R}^n$. Subsequently, problem \eqref{ParametricOptimization} will be said to be convex-concave if the functions $f$ and $g_i$, $i=1, \ldots, p$, are concave-convex. Similarly, problem \eqref{ParametricOptimization} will just be said to be convex if the latter functions are convex w.r.t. $y$. We will also use the \emph{Mangasarian-Fromovitz constraint qualification} (\emph{MFCQ})
 \begin{equation}\label{MFCQ}
 \left.
 \begin{array}{r}
   \nabla_y g(\bar x, \bar y)^\top u=0\\
   u\geq 0, \;g(\bar x, \bar y)\leq 0, \; u^\top  g(\bar x, \bar y)=0
 \end{array}
 \right\} \Longrightarrow u=0
 \end{equation}
 and the \emph{linear independence constraint qualification} (\emph{LICQ})
 \begin{equation}\label{LICQ}
   \sum_{i\in I(\bar x, \bar y)}u_i\nabla_y g_i(\bar x, \bar y)=0 \;\; \Longrightarrow \;\; u_i =0, \;\; i\in I(\bar x, \bar y),
 \end{equation}
 where $I(\bar x, \bar y):=\left\{i=1, \ldots, p|\;  g_i(\bar x, \bar y)=0\right\}$. It is well-known that if the LICQ holds at $(\bar x, \bar y)$, then the MFCQ will automatically hold at the same point.
\begin{thm}\label{subdifferentialEstimate}{\em  Considering the optimal value function \eqref{OptValueFn}, it holds that:
\begin{itemize}
\item[(i)] If $Y$ defines a compact set such that $K(x):=Y$ for all $x\in \mathbb{R}^n$, then for all $x\in \mathbb{R}^n$,
 \begin{equation}\label{copartialphi}
\bar{\partial} \varphi(x) = \mbox{co}\left\{\nabla_x f(x, y)\left| \;\; y\in S(x) \right.\right\}.
\end{equation}
 If additionally, $Y$ is convex and $f$ concave-convex, then for all $x\in \mathbb{R}^n$,
\begin{equation}\label{copartialphi-no co}
 \bar{\partial} \varphi(x) = \left\{\nabla_x f(x, y)\left| \;\; y\in S(x) \right.\right\}.
\end{equation}
 \item[(ii)] Let problem \eqref{ParametricOptimization} be convex and $S$ \eqref{S(x)} single-valued (i.e., $S:=s$) around $\bar x$. Furthermore, let $\mbox{gph }K$ be compact and the MFCQ hold at $(\bar x, \bar y)$ with $\bar y = s(\bar x)$. Then  for all $x$ near $\bar x$, with $s(x):=y$, it holds that
  \begin{equation}\label{Sub-varphi}
   \bar{\partial} \varphi(x) = \underset{u\in \Lambda (x,y)}\bigcup\; \left\{\nabla_x f(x,y) + \nabla_x g(x,y)^\top u \right\}.
  \end{equation}
  \item[(iii)] If $\mbox{gph }K$ is compact and the LICQ holds at $(x, y)$, for all $y\in S(x)$, then  the following equality holds around $x$, with $u=\lambda(x, y)$:
  \begin{equation}\label{subphi-KM-II}
   \bar{\partial} \varphi(x) = \mbox{ co }\underset{y\in S(x)}\bigcup \; \left\{\nabla_x f(x,y) + \nabla_x g(x,y)^\top u \right\}.
  \end{equation}
\end{itemize}
}\end{thm}
\begin{proof}
(i) Equality \eqref{copartialphi} is a well-known result by Danskin \cite{DanskinBook1967}. 
As for \eqref{copartialphi-no co}, the maximization case proven in \cite{bernhard1995theorem} can easily be adapted to our minimization case in  \eqref{OptValueFn}. 
(ii) As the MFCQ holds at $(\bar x, y)$ with $\bar y= s(\bar x)$, and remains persistent in some neighborhood of this point, it is well-known that the formula \eqref{Sub-varphi} will hold in some neighborhood of $\bar x$, given that $\mbox{gph }K$ is compact; see, e.g.,  \cite{ShimizuIshizukaBardBook1997}.
(iii) Start by noting that as in the previous case, the LICQ being persistent near $(\bar x, y)$ for all $y \in S(\bar x)$, where it holds, then we have \eqref{subphi-KM-II} from Gauvin and Dubeau \cite{GauvinDubeau1982} given that  $\mbox{gph }K$ is compact.
\end{proof}
It is important to recall that the compactness assumption on $\mbox{gph }K$ can be relaxed by instead imposing some set-valued-type continuity properties on $S$ \eqref{S(x)}; see, e.g., \cite{BonnansShapiroBook2000,  GauvinDubeau1982, MordukhovichNamPhanVarAnalMargBlP,MordukhovichNamYenSubgradients2009,ShimizuIshizukaBardBook1997}. But for the purpose of simplifying the framework used in this paper, we do not consider such relaxations here. However, most of the developed results will remain valid under such assumptions.
It is also important to mention that the generalized Hessian constructions developed in this paper can rely on any exact formulas of the subdifferential of $\varphi$; exact formulas for subdifferentials of $\varphi$ based on intersection operators can be found in \cite[Theorem 5.1d]{Benko2021} or \cite[Chapter 6]{ShimizuIshizukaBardBook1997}, for example. But, of course, depending on the expression of $\bar\partial \varphi(x)$ used, we might need different requirements and the resulting estimates might also be different.

It is clear from Theorem \ref{subdifferentialEstimate} that the subdifferential of $\varphi$ is a ``function'' of the Lagrange multipliers and optimal solution set-valued mappings. It is therefore natural to imagine that second order subdifferentials for $\varphi$ can primarily be constructed based on the generalized differentiation tools for this mappings. Hence, to get well prepared for our main results in the next sections, we first provide some useful properties of these mappings here. We start with a coderivative estimate for $\Lambda$ and deduce a condition ensuring that this mapping is \emph{Lipschitz-like}.
From here on, we will also assume that the graph of the set-valued mapping $\Lambda$ \eqref{Lambdadefined} is nonempty. Given that $S(x)\neq \emptyset$ for all $x\in \mathbb{R}^n$, the latter is automatically satisfied if there exists a point $(x,y)\in \mbox{gph }S$, where a constraint qualification, e.g., the MFCQ or LICQ, holds. Also recall that for a point $(\bar{x} , \bar{y} , \bar{u})\in \mbox{gph }\Lambda$, we can define  the standard partition
\begin{equation}\label{multiplier sets}
\begin{array}{rllll}
 \eta  &:=& \eta(\bar{x} , \bar{y} , \bar{u}) &:=& \left\{i=1, \ldots, p\,|\; \bar{u}_i =0, \, g_i(\bar{x} , \bar{y} )<0\right\},\\
\theta   &:=& \theta(\bar{x} , \bar{y} , \bar{u}) &:=& \{i=1, \ldots, p\,|\; \bar{u}_i =0, \, g_i(\bar{x} , \bar{y} )=0\},\\
\nu   &:=&  \nu(\bar{x} , \bar{y} , \bar{u}) &:=& \{i=1, \ldots, p\,|\; \bar{u}_i >0, \, g_i(\bar{x} , \bar{y} )=0\},
\end{array}
\end{equation}
of the indices of the constraints of the feasible set of problem \eqref{ParametricOptimization}. This allows us to introduce the following special class of multipliers, which permits an elegant presentation of the remaining results of this section:
\begin{equation}\label{yameogo}
\mho(\bar x, \bar y, \bar u, u^*):=\left\{(a,c)\left|\begin{array}{r}
                                                     u^*_\nu + \nabla_yg_\nu(\bar x,\bar y)a=0, \; c_\eta=0\\
                                                     \forall i\in \theta: \; \left(u^*_i + \nabla_yg_i(\bar x, \bar y)a >0\,\wedge c_i>0\right)\\
                                                     \qquad\qquad \qquad\vee c_i \left(u^*_i + \nabla_yg_i(\bar x, \bar y)a\right)=0                                              \end{array}\right.\right\}
\end{equation}
with $(\bar{x} , \bar{y} , \bar{u})\in \mbox{gph }\Lambda$ and $u^*\in \mathbb{R}^p$.
\begin{prop}\label{gouola}{\em Consider a point $(\bar{x} , \bar{y} , \bar{u})\in \mbox{gph }\Lambda$  and suppose that we have
\begin{equation}\label{CQfortheThing}
    \left.
\begin{array}{r}
                                               \nabla^2_{yx}L(\bar x, \bar y, \bar u)a +  \nabla_xg(\bar x, \bar y)^\top c=0\\
                                               \nabla^2_{yy}L(\bar x, \bar y, \bar u)a +  \nabla_yg(\bar x, \bar y)^\top c=0\\
                                                (a, c)\in \mho(\bar x, \bar y, \bar u, 0)
                                                          \end{array}
\right\} \Longrightarrow \left\{\begin{array}{l}
                                  a=0,\\
                                  c=0.
                                \end{array}
 \right.
\end{equation}
Then for all $u^*\in \mathbb{R}^p$, we have the following upper estimate:
\begin{equation}\label{Delanoe}
\begin{array}{l}
  D^*\Lambda(\bar x, \bar y| \bar u)(u^*)  \subseteq \\[1ex]
                                          \qquad \qquad  \left\{\left.\left[\begin{array}{c}
                                               \nabla^2_{yx}L(\bar x, \bar y, \bar u)a +  \nabla_xg(\bar x, \bar y)^\top c\\
                                               \nabla^2_{yy}L(\bar x, \bar y, \bar u)a +  \nabla_yg(\bar x, \bar y)^\top c
                                                          \end{array}
 \right]\right| (a, c)\in \mho(\bar x, \bar y, \bar u, u^*) \right\}.
\end{array}
\end{equation}
Furthermore, $\Lambda$ is Lipschitz-like around $(\bar x, \bar y, \bar u)$, provided we also have
\begin{equation}\label{holala}
\begin{array}{l}
(a, c)\in \mho(\bar x, \bar y, \bar u, 0)  \Longrightarrow \left\{\begin{array}{l}
                                  \nabla^2_{yx}L(\bar x, \bar y, \bar u)a +  \nabla_xg(\bar x, \bar y)^\top c=0,\\
                                               \nabla^2_{yy}L(\bar x, \bar y, \bar u)a +  \nabla_yg(\bar x, \bar y)^\top c=0.
                                \end{array}
 \right.
\end{array}
\end{equation}
}\end{prop}
\begin{proof}Start by quickly recalling that the set-valued mapping $\Lambda$ is closed, given that all the functions involved in problem \eqref{ParametricOptimization} are assumed to continuously differentiable. Furthermore, it can be written as
\begin{equation}\label{LambdaRefor}
    \begin{array}{l}
\Lambda(x,y)= \left\{u\in \mathbb{R}^p|\; \psi(x,y,u)\in \Xi\right\}\\
\mbox{with }\left\{\begin{array}{l}
                      \psi(x,y,u):= \left[\nabla_yL(x,y,u)^\top, u^\top, -g(x,y)^\top \right],\\
                      \Xi:= \{0_m\}\times \Theta,\\
                      \Theta:=\left\{(a,b)\in \mathbb{R}^{p}\times \mathbb{R}^{p}|\; a\geq 0, \; b\geq 0, \; a^\top b=0\right\}.
                   \end{array}
\right.
\end{array}
\end{equation}
Let $(x^*, y^*)\in D^*\Lambda(\bar x, \bar y| \bar u)(u^*)$. Then, by the definition of the concept of coderivative \eqref{cod-definition}, $(x^*, y^*, -u^*)\in N_{\mbox{gph }\Lambda}(\bar x, \bar y, \bar u)$. Hence, it follows from Theorem \ref{normal cone estimate} that there exists a vector $(a, b, c)\in N_{\Xi}(\psi(\bar x, \bar y, \bar u))$ such that
\begin{align}
\left[\begin{array}{r}
                                                            x^*\\
                                                            y^*\\
                                                            -u^*
                                                          \end{array}
 \right] =\nabla \psi(\bar x, \bar y, \bar u)^\top\left[\begin{array}{r}
                                                            a\\
                                                            b\\
                                                            c
                                                          \end{array}
 \right] = \left[\begin{array}{l}
                                                            \nabla^2_{yx}L(\bar x,\bar y,\bar u)a -  \nabla_xg(\bar x,\bar y)^\top c\\
                                                            \nabla^2_{yy}L(\bar x,\bar y,\bar u)a -  \nabla_yg(\bar x,\bar y)^\top c\\
                                                             \nabla_yg(\bar x,\bar y) a + b
                                                          \end{array}
 \right],\label{medong}
\end{align}
provided that the counterpart of \eqref{Basic CQ2} holds at $(\bar x, \bar y, \bar u)$.  For the latter requirement and the finalization of the estimate in \eqref{Delanoe}, note that
\[
                                 N_{\Xi}(\psi(\bar x, \bar y, \bar u))=\mathbb{R}^m\times \left\{\begin{array}{lll}
                                                                             & u^*_i=0 & \forall i \in \nu \\
                                                                            (u^*,v^*)\in\mathbb{R}^{p}\times \mathbb{R}^{p}: & v^*_i=0 & \forall i \in \eta\\
                                                                             & (u^*_i<0 \wedge v^*_i<0)\, \vee\, u^*_iv^*_i=0 & \forall i\in \theta
                                                                          \end{array}\right\}.
\]
Combining this equality with  \eqref{medong}, one can easily check that the counterpart of \eqref{Basic CQ2} is satisfied under assumption  \eqref{CQfortheThing}. As for the Lipschitz-likeness of $\Lambda$ around $(\bar x, \bar y)$, observe from the discussion above that
\begin{equation}\label{Delanoe}
\begin{array}{l}
  D^*\Lambda(\bar x, \bar y| \bar u)(0)  \subseteq \\[1ex]
                                          \qquad \qquad  \left\{\left.\left[\begin{array}{c}
                                               \nabla^2_{yx}L(\bar x, \bar y, \bar u)a +  \nabla_xg(\bar x, \bar y)^\top c\\
                                               \nabla^2_{yy}L(\bar x, \bar y, \bar u)a +  \nabla_yg(\bar x, \bar y)^\top c
                                                          \end{array}
 \right]\right| (a, c)\in \mho(\bar x, \bar y, \bar u, 0) \right\}.
\end{array}
\end{equation}
Hence, under \eqref{holala}, $D^*\Lambda((\bar x, \bar y)| \bar u)(0)=\{0\}$. This ensures that $\Lambda$ is Lipschitz-like around $(\bar x, \bar y, \bar u)$, based on the Mordukhovich coderivative criterion \eqref{coderivative criterion}.
\end{proof}
\begin{rem}\label{rouena}
The conclusions of Proposition \ref{gouola} remain valid if the multipliers set $\mho(\bar x, \bar y, \bar u, u^*)$ in \eqref{yameogo} is replaced by the following one:
\begin{equation}\label{yameogoC}
\overline{\mho}(\bar x, \bar y, \bar u, u^*):=\left\{(a, c)\left|\begin{array}{r}
                                                     u^*_\nu + \nabla_yg_\nu(x,y)a=0, \; c_\eta=0\\
                                                     \forall i\in \theta: \; c_i \left(u^*_i + \nabla_yg_i(x,y)a\right)\geq 0                                              \end{array}\right.\right\}.
\end{equation}
The implications in \eqref{CQfortheThing} and \eqref{holala} correspond to M(or \emph{Mordukhovich})-type conditions while \eqref{Delanoe} can be labeled as M-type estimate of the coderivative of $\Lambda$. Similarly, with \eqref{yameogoC}, we will respectively have C(or \emph{Clarke})-type conditions and a C-type estimate for the coderivative of  $\Lambda$. More details on constructions and vocabulary in this vein can be found, for example, in \cite{DempeMordukhovichZemkohoTwo-level,DempeZemkohoOnTheKKTRef}. It is also important to recall that the result in Proposition \ref{gouola} can be obtained if we replace condition \eqref{CQfortheThing} by the weaker calmness of the  set-valued mapping
\[
{ {\Psi(v):= \left\{(x, y, u)\in \mathbb{R}^n\times \mathbb{R}^m\times \mathbb{R}^p\left|\; \psi(x,y,u) + v\in \Xi\right.\right\},}}
\]
where the function $\psi$ and the set $\Xi$ are defined in \eqref{LambdaRefor}; see, e.g., \cite{Benko2021, HenrionEtAL2002}.
\end{rem}

Next, we provide a simple, yet powerful relationship between the coderivatives of $S$ and $\Lambda$, that will allow the derivation of a complete estimate of the former based on Proposition \ref{gouola}.
\begin{prop}\label{teclair}
{\em Suppose that the functions $f(x, .)$ and $g_i(x, .)$, $i=1, \ldots, p$, are convex, for all $x\in \mathbb{R}^n$, and the MFCQ holds at $(\bar x, \bar y)\in \mbox{gph }S$. Then for all $y^*\in \mathbb{R}^m$,
\begin{equation}\label{CodSNO}
    D^*S(\bar x|\bar y)(y^*) \subseteq \underset{u\in \Lambda (\bar x, \bar y)} \bigcup  \left\{x^*\in \mathbb{R}^n| \; (x^*, -y^*)\in D^*\Lambda(\bar x, \bar y| u)(0) \right\}.
\end{equation}
If additionally, condition \eqref{CQfortheThing} holds at $(\bar x, \bar y, u)$ for all $u\in \Lambda (\bar x, \bar y)$, then for all $y^*\in \mathbb{R}^m$,
\begin{equation}\label{gouene1}
\begin{array}{rll}
  D^*S(\bar x|\bar y)(y^*)  & \subseteq & \underset{u\in \Lambda (\bar x, \bar y)} \bigcup  \;\;\underset{(a, c)\in \mho(\bar x, \bar y, u, 0)} \bigcup \left\{\nabla^2_{xy}L(\bar x, \bar y,  u)^\top a +  \nabla_xg(\bar x, \bar y)^\top c\right|\\
                                     &  &\qquad \qquad \qquad \qquad  \left. y^* + \nabla^2_{yy}L(\bar x, \bar y, u)a +  \nabla_yg(\bar x, \bar y)^\top c =0 \right\}.
\end{array}
\end{equation}
}
\end{prop}
\begin{proof}
Under the assumption that the functions $f(x, .)$ and $g_i(x, .)$, $i=1, \ldots, p$, are convex, for all $x\in \mathbb{R}^n$ and the MFCQ holds at $(\bar x, \bar y)\in \mbox{gph }S$, it follows that near this point, the optimal solution set-valued mapping can take the form
$$
\begin{array}{l}
S(x)= \Pi_1 \circ Q(x)\\
\mbox{with }\; \left\{
\begin{array}{l}
\Pi_1(y,u):= y,\\
Q(x):= \left\{(y,u)|\; u\in \Lambda(x,y)\right\}.
\end{array}
\right.
\end{array}
$$
Observe from the definition of the set-valued map $Q$ and the function $\Pi_1$ that we have $\nabla \Pi_1(y,u)^\top y^* = \left[(y^*)^\top, 0^\top_p\right]^\top$ and $(y, u)\in Q(\bar x)\cap \Pi^{-1}_1(\bar y)$ if and only if $u\in \Lambda (\bar x, \bar y)$. The set-valued mapping $\Lambda$ is closed, given that the functions $f$ and $g$ are assumed to be continuously differentiable throughout the paper.
 Then applying the chain rule from Theorem \ref{chain rule} to the above expression of $S$,
\begin{equation}\label{yespala}
    D^*S(\bar x|\bar y)(y^*) \subseteq \underset{u\in \Lambda (\bar x, \bar y)} \bigcup D^*Q(\bar x| \bar y, u)(y^*, 0),
\end{equation}
provided that the set-valued  mapping
$
M(x,y):=\left\{(z,u)|\; (z,u)\in Q(x), \; \Pi_1 (z,u)=y  \right\}
$
is locally bounded around $(\bar x, \bar y)$. The latter is indeed true and can easily be shown. 
On the other hand, we have
\begin{equation}\label{beaucoup}
\begin{array}{rll}
 D^*Q(\bar x| \bar y, u)(y^*, u^*) & = & \left\{x^*| \; (x^*, -y^*, -u^*)\in N_{\mbox{gph }Q}(\bar x, \bar y, u) \right\}\\
                              & = & \left\{x^*| \; (x^*, -y^*, -u^*)\in N_{\mbox{gph }\Lambda}(\bar x, \bar y, u) \right\}\\
                              & = & \left\{x^*| \; (x^*, -y^*)\in D^*\Lambda(\bar x, \bar y| u)(u^*) \right\}
\end{array}
\end{equation}
considering the fact that the graph of $Q$ is the same as that of $\Lambda$. Combining the last equality in \eqref{beaucoup} with inclusion \eqref{yespala}, we have \eqref{CodSNO}. As for the Lipschitz-like property of $S$ at $(\bar x, \bar y)$, this is based on inclusion \eqref{CodSNO}, while applying the coderivative criterion \eqref{coderivative criterion}, given that $S$ is a closed set-valued mapping.
\end{proof}
\begin{rem}
Suppose that the functions $f(x, .)$ and $g_i(x, .)$, $i=1, \ldots, p$, are convex, for all $x\in \mathbb{R}^n$, and the MFCQ holds at $(\bar x, \bar y)\in \mbox{gph }S$. If in addition, the set-valued mapping $S$ is closed, then $S$ is Lipschitz-like around $(\bar x, \bar y)$ provided that  $\Lambda$ is Lipschitz-like around $(\bar x, \bar y, u)$, for all $u\in \Lambda(\bar x, \bar y)$. It is clear from Proposition \ref{gouola} that the latter holds if the qualification condition \eqref{holala} holds at $(\bar x, \bar y, u)$, for all $u\in \Lambda(\bar x, \bar y)$.
\end{rem}

This estimate of the coderivative of $S$  in \eqref{gouene1} was obtained in \cite[Theorem 4.3]{MordukhovichOutrataCoderivativeOfQuasi-Var2007} using a different approach, which relies on the computation of the coderivative of a certain normal cone map. Also note that the mapping in the latter reference is slight more general and the assumption framework is based on calmness constructions. Our results are amenable to more general settings, and corresponding calmness based constructions are possible; see some related discussions in the next sections.


Coming towards the end of this section, it is important to emphasize that the Lipschitz-like property for $\Lambda$ \eqref{Lambdadefined} or $S$ \eqref{S(x)}, in the sense of set-valued mappings, is not necessary to achieve the principal goal of this paper. The main condition in the process is \eqref{CQfortheThing}, which allows us to estimate the coderivative of the relevant map. Hence, let us provide a small example to illustrate the condition \eqref{CQfortheThing}.
\begin{example} Consider the value function $\varphi$ \eqref{ParametricOptimization} with $f$ and $g$ respectively defined by
$$
\begin{array}{lll}
f(x,y):= \left(y_1 - x_1 +20\right)^2 + \left(y_2 - x_2 +5\right)^2 & \mbox{and}& g(x,y):= \left(\begin{array}{c}
                                                                                         2y_1 - x_1 +10\\
                                                                                         2y_2 - x_2 +10\\
                                                                                         - y_1 - 10\\
                                                                                         - y_2 - 10\\
                                                                                          \;\;\, y_1 - 20\\
                                                                                          \;\;\, y_2 - 20
                                                                                      \end{array}\right).
\end{array}
$$
Considering the reference point $(\bar x, \bar y)$ with $\bar x := (25, 30)$ and $\bar y :=(5, 10)$, we can easily check that the vector $\bar u$, with second component $5$ and zero everywhere else, is a corresponding Lagrange multiplier. Furthermore, at  $(\bar x, \bar y, \bar u)$ , $\eta=\{1, 3, 4, 5, 6\}$, $\theta=\emptyset$, and $\nu:=\{2\}$. With this configuration of the index sets, we can easily check that
$$
\overline{\mho}(\bar x, \bar y, \bar u, 0)=\left\{(a, c)\in \mathbb{R}^2 \times \mathbb{R}^6|\;\, a_2=0, \; c_1=0, \; c_3=0, \; c_4=0, \; c_5=0, \; c_6=0\right\}.
$$
Hence, to show that condition \eqref{CQfortheThing} holds, it suffices to demonstrate that $a_1=0$ and $c_2=0$ for any vector $(a,c)$ verifying the left-hand-side of the implication. To proceed, observe that in the context of this example, equation
$
\nabla^2_{yx}L(\bar x, \bar y, \bar u)a +  \nabla_xg(\bar x, \bar y)^\top c=0
$
is equivalent to having the conditions $2a_1 + c_1 =0$ and  $2a_2 + c_2 =0$ simultaneously satisfied.  Therefore, combining this with the fact $(a,c)\in \overline{\mho}(\bar x, \bar y, \bar u, 0)$, the result follows.
\end{example}

Finally, to close this section, we would like to mention that we will occasionally require that the set-valued mapping $\Lambda$ or $S$ is locally single-valued and Lipschtz continuous. We do not specifically address this topic in this paper, as it is out of the scope of our work. But it is important to recall that in the case of $S$, many publications have addressed this question; see Subsections \ref{Single-valued optimal solution map} and \ref{Single-valued solution and multipliers maps}, where more details and references are provided. As for $\Lambda$, it is closely related to
$$
S_{KKT}(x):=\left\{(y,u)\in \mathbb{R}^m \times \mathbb{R}^p|\; u\in \Lambda(x,y)\right\},
$$
which has been widely studied in the literature; e.g., \cite{Robinson1981,dontchev1996characterizations}. Obviously, the coderivative calculations will easily show some close interplays between the two mappings, similarly, though much simpler than, to those between $S$ and $\Lambda$ provided in Proposition \ref{teclair}. More precisely, note that it is obvious that $\Lambda$ can be written as
$$
  \Lambda(x,y):=\left\{u\in \mathbb{R}^p|\; 0\in \psi(x,y,u) + \Phi(u)  \right\},
 $$
where the function $\psi$ and the set-valued mapping $\Phi$ are respectively defined by
 $\psi(x,y,u):= -\left[\nabla_y L(x,y,u)^\top, g(x,y)^\top\right]^\top$  and $\Phi(u):=\{0_m\}\times N_{\mathbb{R}^p_+}(u)$.
Many papers have been devoted to conditions ensuring that general mappings of the above form are locally single-valued and Lipschitz continuous; see, e.g., \cite{DontchevWilliam1994, Robinson1981}.
\section{Generalized Hessian estimates in the presence of convexity}\label{Generalized Hessian estimates}
\subsection{Case where the feasible set is unperturbed}\label{Case where the feasible set is unperturbed}
In this subsection, we assume that the feasible set of problem \eqref{OptValueFn} is independent of $x$.  More precisely, for $g: \mathbb{R}^m \rightarrow \mathbb{R}^p$, our attention here will be on the following function assumed to be finite-valued:
\begin{equation}\label{Danskin value function}
    \varphi(x):=\underset{y\in Y}\min~f(x,y) \; \mbox{ with }\; Y:=\left\{y\in \mathbb{R}^m|\; g(y)\leq 0\right\}.
\end{equation}
\begin{thm}\label{partial2phi} {\em If $Y$ is a convex and compact set and $f$ is concave-convex funtion, then for $\underline{x}\in \partial \varphi(\bar x)$ and $\underline{x}^*\in \mathbb{R}^n$, we have \eqref{man-no-lap3}. If additionally, the MFCQ \eqref{MFCQ} holds at $y$ and  $\nabla^2_{yx}f(\bar x, y)$ is full rank, for all $y\in S(\bar x)$ with $\nabla_x f(\bar x, y)=\underline{x}$, then for  $\underline{x}^*\in \mathbb{R}^n$,
\begin{equation}\label{Trieveler}
\begin{array}{l}
  \partial^2 \varphi(\bar x| \underline{x})(\underline{x}^*) \qquad \subseteq \qquad \underset{y\in S(\bar x): \;\, \underline{x}=\nabla_x f(\bar x, y)}\bigcup\;\; \underset{u\in \Lambda (\bar x, y)} \bigcup  \;\;\underset{(a, c)\in \mho(\bar x, y, u, 0)} \bigcup \\
                                                              \left\{\nabla^2_{xx}f(\bar x, y)\underline{x}^* + \nabla^2_{yx}f(\bar x, y)^\top a\right|
                                     \left. \nabla^2_{xy}f(\bar x, y)\underline{x}^* + \nabla^2_{yy}L(\bar x, y, u)a +  \nabla g(y)^\top c =0 \right\},
\end{array}
\end{equation}
where the sets  $\Lambda (\bar x, y)$ and $\mho(\bar x, y, u, 0)$ are defined in \eqref{Lambdadefined} and \eqref{yameogo}, respectively.
}
\end{thm}
\begin{proof}
Note that under the assumptions of the theorem, we have equality \eqref{copartialphi-no co} from Theorem \ref{subdifferentialEstimate}. This equality can equivalently be written as
$
    \partial \varphi(x) = \nabla_x f\circ\Psi(x),
$
where $\Psi(x):= \{x\}\times S(x)$. Further observe that the set-valued map $\Psi$ is closed, given that the counterpart of $S$ \eqref{S(x)} for \eqref{Danskin value function} can take the form $S(x):=\left\{y\in Y|\;  f(x,y)=\varphi(x)\right\}$, and is thus closed as $\varphi$ is locally Lipschitz continuous under the imposed continuous differentiability of the function $f$ and compactness of the set $Y$.
Also note that the set-valued mapping
$
M(x,z):= \left\{(a, b)|\, a=x, \; b\in S(x), \;\nabla_x f(a, b)=z\right\}
$
 is locally bounded around $(\bar x, \underline{x})$, given that
$
M(X\times Z) \subseteq X \times Y,
$
for some neighborhoods $X$ and $Z$ of $\bar x$ and $\underline{x}$, respectively, with assumed $X$ to be bounded.
Hence, we have
\begin{equation}\label{Hessian phi 1}
\partial^2 \varphi(\bar x| \underline{x})(\underline{x}^*) \subseteq \underset{y\in S(\bar x): \;\, \underline{x}=\nabla_x f(\bar x, y)}\bigcup\left[D^*\Psi(\bar x|\bar x, y)\left(\left[ \begin{array}{l}
                         \nabla^2_{xx}f(\bar x, y)\underline{x}^* \\
                          \nabla^2_{yx}f(\bar x, y)^\top\underline{x}^*
                       \end{array}
  \right]\right)  \right]
\end{equation}
from the chain rule in Theorem \ref{chain rule}. Finally, for the right-hand-side of \eqref{Hessian phi 1}, it follows that since $S$ is closed as shown above, applying Theorem \ref{pp1} to  $\Psi$ leads to
\begin{equation}\label{Cod Psi3}
    \begin{array}{l}
D^*\Psi\left(\bar x \left|\bar x, y\right.\right)\left(\left[ \begin{array}{l}
                         \nabla^2_{xx}f(\bar x, y)\underline{x}^* \\
                          \nabla^2_{yx}f(\bar x, y)^\top\underline{x}^*
                       \end{array}
  \right]\right) \subseteq \nabla^2_{xx}f(\bar x, y)\underline{x}^*  + D^*S(\bar x|y)\left(\nabla^2_{yx}f(\bar x, y)^\top\underline{x}^*\right)
\end{array}
\end{equation}
 given that the corresponding counterpart of qualification condition \eqref{QC-Cod} is automatically satisfied, as $0\in D^*S(\bar x|\bar y)(0)$ by the positive homogeneity of the coderivative mapping. Finally, inclusion \eqref{man-no-lap3} is obtained by combining \eqref{Hessian phi 1} and \eqref{Cod Psi3}. As for inclusion \eqref{Trieveler}, it obviously follows from \eqref{gouene1}. One can easily check that if the matrix $\nabla^2_{yx}f(\bar x, y)$ is full rank, then the qualification condition \eqref{CQfortheThing} holds.
\end{proof}
The tools involved in the calculations in this results are standard and can easily be checked; note in particular that requiring the matrix $\nabla^2_{yx}f(\bar x, y)$ is full rank helps to ensure that condition \eqref{CQfortheThing} is satisfied.  To illustrate this result, we consider the following example with a linear program having left-hand-side perturbation.
\begin{example}{
Consider the optimal value function
$$
\varphi(x):=\underset{y}\min~\left\{x^\top y|\; Ay\leq b\right\},
$$
  where $Y:=\left\{y\in \mathbb{R}^m|\; Ay\leq b\right\}$ is compact and $A$ is a full rank matrix.  For a couple $(\bar x, \underline{x})$, we have $\underline{x}\in \partial \varphi (\bar x)$ if and only if $\underline{x}\in S (\bar x)$. We also have  $\bar u=\lambda(\bar x, \underline{x})$, as $A$ is full rank. For $(\bar x, \underline{x}, \bar u)$, consider the definitions of $\nu$, $\eta$, and $\theta$ given in \eqref{multiplier sets} for the constraint set $Y$, let
 $$
\begin{array}{c}
  \Xi(A):= \big\{(a,c)|\;A_{\nu}a=0, \; c_{\eta}=0, \; \forall i\in \theta: \;\left(A_ia>0 \wedge c_i >0\right)\, \vee \,  c_i \left(A_i a\right)=0 \big\}.
\end{array}
$$
Further note that the qualification condition \eqref{CQfortheThing} holds at the point  $(\bar x, \underline{x}, \bar u)$, as $A$ has full rank.  It therefore follows from inclusion \eqref{Trieveler} that
$$
\partial^2\varphi(\bar x|\underline{x})(\underline{x}^*)  \subseteq \left\{a|\; (a, c)\in \Xi(A), \; A^\top c + \underline{x}^* =0\right\}.
$$
}\end{example}
To conclude this subsection, recall that different structures are possible for the set $\mho(\bar x, y, u, 0)$, which depends on the reformulation of the complementarity conditions defined in $\Lambda(\bar x, y)$, as discussed in Remark \ref{rouena}.


\subsection{Single-valued optimal solution map}\label{Single-valued optimal solution map}
We assume here that the optimal solution mapping $S$ \eqref{S(x)} is single-valued; i.e., $S:=s$. On the other hand, we let $\Lambda$ \eqref{Lambdadefined} be set-valued. We can then estimate the generalized Hessian of $\varphi$ \eqref{OptValueFn} as follows.

\begin{thm}\label{foulam}{\em Suppose that $\mbox{gph }K$ \eqref{K(x)}
is compact and  $S$ is single-valued (i.e., $S:=s$)  and Lipschitz continuous around $\bar x$. Furthermore, assume that problem \eqref{ParametricOptimization} is convex and the MFCQ holds at $(\bar x, \bar y)$ with $\bar y = s(\bar x)$ and the  qualification condition
\begin{equation}\label{QCDUAL}
    \left.
    \begin{array}{r}
     -a^*\in \partial \langle b^*, s\rangle (\bar x)\\
     (a^*, b^*)\in D^*\Lambda(\bar x, \bar y| u)(0)
    \end{array}
    \right\} \Longrightarrow \left\{\begin{array}{l}
                                     a^*=0\\
                                     b^*=0
                                    \end{array}
    \right.
\end{equation}
holds at $(\bar x, \bar y,  u)$ with $\bar y=s(\bar x)$ for any $u\in \Lambda(\bar x, \bar y)$. Then, for $\underline{x}\in \partial \varphi(\bar x)$ and any $\underline{x}^*\in \mathbb{R}^n$, we have the following estimate for the second order subdifferential of $\varphi$:
\begin{align}
\partial^2\varphi(\bar x|\underline{x})(\underline{x}^*) \subseteq  &  \underset{u\in \Lambda(\bar x, \bar y)}\bigcup \;\;\underset{(\zeta^*_x, \zeta^*_y)\in D^*\Lambda(\bar x, \bar y|  u)\left(\nabla_x g(\bar x, \bar y)\underline{x}^*\right)}\bigcup \left. \Big\{\nabla^2_{xx}L(\bar x, \bar y, u)\underline{x}^* + \zeta^*_x  \right.\nonumber\\
  & \qquad \qquad \qquad \qquad \qquad +\;\partial \left\langle\zeta^*_y + \nabla^2_{xy}L(\bar x, \bar y,  u)\underline{x}^*, \;\, s\right\rangle(\bar x)\Big\}.\label{Yespapa}
\end{align}
If additionally, we suppose that conditions \eqref{CQfortheThing} and \eqref{holala} are satisfied at $(\bar x, \bar y, u)$ for all $u\in \Lambda (\bar x, \bar y)$ with $\bar y=s(\bar x)$,   then the qualification condition \eqref{QCDUAL} holds at the relevant points and for $\underline{x}\in \partial \varphi(\bar x)$ and  $\underline{x}^*\in \mathbb{R}^n$, we have
\begin{equation}\label{Vanerone}
\begin{array}{l}
\partial^2\varphi(\bar x|\underline{x})(\underline{x}^*) \subseteq    \underset{u\in \Lambda(\bar x, \bar y)}\bigcup \;\;\underset{(a, c)\in \mho\left(\bar x, \bar y, u, \nabla_x g(\bar x, \bar y)\underline{x}^*\right)}\bigcup\\
 \qquad \qquad  \Big\{\nabla^2_{xx}L(\bar x, \bar y, u)\underline{x}^* + \nabla^2_{yx}L(\bar x, \bar y, u)a +  \nabla_xg(\bar x, \bar y)^\top c\\
   \qquad \qquad \qquad \qquad    + \partial \left\langle \nabla^2_{xy}L(\bar x, \bar y,  u)\underline{x}^* + \nabla^2_{yy}L(\bar x, \bar y, u)a +  \nabla_yg(\bar x, \bar y)^\top c, \;\, s \right\rangle(\bar x)\Big\}.
\end{array}
\end{equation}
}\end{thm}
\begin{proof}
 Obviously, under the assumptions made, we have equality \eqref{Sub-varphi}, which can be rewritten as  $\partial \varphi(x)  = \nabla_xL \circ \Phi \circ \psi (x)$ with  $\psi (x):=\left[x^\top, s(x)^\top \right]^\top$
 and $\Phi(a,b)   := \{(a,b)\}\times \Lambda(a,b)$. Consider the set-valued mapping
 $
 M(x,z):=\left\{(a, b, c)\in \Phi  \circ\psi (x)|\; \nabla_xL(a, b, c)=z \right\}
 $
 and a sequence $\{(x^k, z^k, a^k, b^k, c^k)\}$ with $(a^k, b^k, c^k)\in M(x^k, z^k)$. Then by definition,
\begin{equation}\label{ayoBaperi}
   \begin{array}{l}
  a^k=x^k, \; b^k=s(x^k), \; z^k= \nabla_x L(a^k, b^k, c^k),\\
  \nabla_y L(a^k, b^k, c^k)=0, \; c^k\geq 0, \; g(a^k, b^k)\leq 0, \; g(a^k, b^k)^\top c^k=0.
\end{array}
\end{equation}
Suppose that $x^k \rightarrow \bar x$, $z^k \rightarrow \underline{x}$ and $\|c^k\|\geq k$ for all $k\in \mathbb{N}$. Then $b^k \rightarrow s(\bar x)$, given that $s$ is assumed to be locally Lipschitz continuous around $\bar x$. We can find a subsequence of $\{c^k\}$, with the same notation, provided there is no confusion, such that $c^k/\|c^k\|$ converges to some $\bar c$ with $\|\bar c\|=1$.
Inserting this subsequence in the second line of \eqref{ayoBaperi} and dividing the terms containing $c^k$ by its norm, we arrive at
\begin{equation*}\label{tafou}
   \nabla_y g(\bar x, \bar y)^\top\bar u=0, \;\; \bar u\geq 0, \;\, g(\bar x, \bar y)\leq 0, \;\, \bar u^\top g(\bar x, \bar y)=0.
\end{equation*}
for the point $(\bar x, \bar y, \bar c)$, as $k \rightarrow \infty$.
Given that the MFCQ holds at $(\bar x, \bar y)$, it follows that we must have $\bar c =0$.
This contradicts the hypothesis that  $\|c^k\|\geq k$ for all $k\in \mathbb{N}$. Hence, $M$ is locally bounded around $(\bar x, \underline{x})$.
Therefore, by the chain rule in Theorem \ref{chain rule},
\begin{equation}\label{HessianPalo}
\partial^2 \varphi(\bar x| \underline{x})(\underline{x}^*) \subseteq \underset{\underset{\underline{x}=\nabla_x L(a, b, u)}{(a, b, u)\in \Phi \circ \psi (\bar x)}}\bigcup\left[D^*(\Phi \circ \psi )(\bar x|a, b, u)\left(\nabla(\nabla_x L)(a, b, u)^\top \underline{x}^*\right)\right]
\end{equation}
given that $\Phi \circ \psi$ is closed, as $s$ is locally Lipschitz continuity around $\bar x$ and the functions $f$ and $g$ are continuously differentiable. Next, note that the set-valued mapping defined by
$
M_{\circ}(x,z):=\left\{(a, b)|\; (a, b)= \psi (x), \; z\in \Phi  (a, b) \right\}
$
is locally bounded around all the points $(\bar x, \bar z)$ with $\bar z\in  \Phi  \circ \psi (\bar x)$ and $\nabla_x L(\bar z)=\underline{x}$, given that $s$ is Lipschitz continuous around $\bar x$ and for some neighborhoods
$X$, $A$, $B$, $C$ of $\bar x$, $\bar a$, $\bar b$, and $\bar c$, respectively, with $X$ being a bounded set while $A$ is a compact set, we have
$$
M(D)= \underset{(x, a, b, c)\in D}\bigcup \left\{(\alpha, \beta)| \alpha=a=x, \beta=b=s(x), c\in \Lambda(x,s(x)  \right\}\subseteq X\times s(X),
$$
where $D:=X\times A\times B\times C$. From the product rule in Theorem \ref{pp1}, it follows that for any $z^*=(x^*, y^*, u^*)$ and some $\bar u\in \Lambda (a, b)$ such that $\bar z=(a,b,\bar u)$, we have
\begin{equation}\label{QEstimate}
 D^*\Phi  (a,b|\bar z) (z^*) \subseteq \left[\begin{array}{c}
                                         x^*\\
                                         y^*
                                       \end{array}
\right] + D^*\Lambda(a,b| \bar u)(u^*),
\end{equation}
given that the counterpart of \eqref{QC-Cod} is automatically satisfied. This obviously leads to
$
D^*\Phi  (a,b|\bar z) (0) \subseteq D^*\Lambda(a,b| \bar u)(0).
$
Then considering the fact that
$$
\begin{array}{c}
\mbox{Ker }D^*\psi (\bar x| a,b) := \left\{(a^*, b^*)|\, 0\in \partial \langle (a^*, b^*), \psi  \rangle (\bar x)  \right\}= \left\{(a^*, b^*)|\, -a^* \in \partial \langle  b^*, s \rangle (\bar x)  \right\},
\end{array}
$$
it is clear that \eqref{QCDUAL} is sufficient for
$
D^*\Phi  (a,b|\bar z)(0) \cap \mbox{Ker }D^*\psi (\bar x| a,b)=\{0\}
$
to hold. Hence, from Theorem \ref{chain rule}, take any $(a, b, u)\in \Phi \circ \psi (\bar x)$ such that $\underline{x}=\nabla_x L(a, b, u)$,
\begin{equation}\label{HessianPalo1}
\begin{array}{l}
  D^*(\Phi  \circ \psi )(\bar x|a, b, c)\left(\nabla(\nabla_x L)(a, b, c)^\top \underline{x}^*\right) \\[1ex]
  \qquad \qquad  \qquad \qquad  \qquad \subseteq \underset{\zeta^*\in D^*\Phi  \left(\bar x, \bar y| \bar x, \bar y, u \right)\left(\nabla(\nabla_x L)(\bar x, \bar y, u)^\top \underline{x}^*\right)}\bigcup \partial \left\langle \zeta^*, \; \psi \right\rangle(\bar x),
\end{array}
\end{equation}
given that $(a, b, u)\in \Phi \circ \psi (\bar x)$ is equivalent to $a=\bar x$, $b=\bar y$ and $u\in \Lambda(\bar x, \bar y)$. Clearly, since
$
\partial \langle (a^*, b^*), \psi  \rangle (\bar x)= a^* + \partial \langle  b^*, s \rangle (\bar x),
$
we have \eqref{Yespapa} from a combination of inclusions \eqref{HessianPalo}, \eqref{QEstimate}, and \eqref{HessianPalo1}. As for inclusion \eqref{Vanerone}, it follows from the insertion of \eqref{Delanoe} in \eqref{Yespapa}. Also note that that \eqref{QCDUAL} is automatically satisfied if condition \eqref{holala} holds at the relevant points; i.e., $(\bar x, \bar y,  u)$ with $\bar y=s(\bar x)$ for any $u\in \Lambda(\bar x, \bar y)$. The latter follows from the fact that the fulfilment of condition \eqref{holala} at $(\bar x, \bar y,  u)$ ensures that we have $D^*\Lambda(\bar x, \bar y|  u)(0)=\{0\}$, by the coderivative criterion \eqref{coderivative criterion}.
\end{proof}

A few remarks are in order to clarify a few things in this result. At first, to guaranty that $S$ is single-valued and locally Lipschitz continuous around $\bar x$, one just needs to additionally impose that the Strong Second Order Sufficient Condition (SSOSC) and the constant rank constraint qualification (CRCQ) are satisfied; cf.   \cite{RalphDempe1995}. Recall that the SSOSC is said to hold at a point $(\bar x, \bar y)$ if for all $u\in \Lambda (\bar x, \bar y)$ and all $d\neq 0$ such that $\nabla_y g_i(\bar x, \bar y)^\top d=0$ if $i\in \eta \cup \nu$, we have
$
d^\top \nabla^2_{yy}L(\bar x, \bar y, u) d >0.
$
As for the CRCQ, it is said to hold at $(\bar x, \bar y)$ if there is a neighborhood $W$ of this point and any subset $I$ of $\{i=1, \ldots, p|\; g_i(\bar x, \bar y)=0\}$, the family of gradients $\{\nabla_y g_i(x, y)|\; (x,y)\in W\}$ has the same rank. To be precise, according to the latter reference, under the MFCQ, CRCQ, and SSOSC,  $s$ is in fact a piecewise continuously differentiable (PC$^1$) function. Furthermore, its Clarke generalized Jacobian can be obtained as
\begin{equation}\label{subs(x)}
\bar{\partial} s(\bar x) = \mbox{co }\left\{s^i(\bar x)|\; i \; \mbox{ s.t. }\; \bar x\in \mbox{cl int Supp }\left(s, s^i\right)  \right\},
\end{equation}
where \emph{cl int} denotes the closure of the interior of $\mbox{Supp }\left(s, s^i\right):=\left\{x|\; s(x)= s^i(x)\right\}$,
 cf. \cite[Chapter 4]{DempeFoundations}. A more detailed expression of \eqref{subs(x)} in terms of the corresponding problem data can be found in the latter reference. Further details on PC$^1$ functions can also be found in \cite{Scholtes1994}.

Secondly, note that a completely detailed upper estimate of $\partial^2\varphi(\bar x|\underline{x})(\underline{x}^*)$ based on  \eqref{Vanerone} is possible. One way to do this is simply to observe that the last term in the right-hand-side of the formula \eqref{Vanerone} is included in the set
$$
\bar{\partial} s(\bar x)^\top \left[\nabla^2_{xy}L(\bar x, \bar y,  u)\underline{x}^* + \nabla^2_{yy}L(\bar x, \bar y, u)a +  \nabla_yg(\bar x, \bar y)^\top c\right],
$$
where $\bar{\partial} s(\bar x)$ represents the Clarke generalized Jacobian of $s$ at $\bar x$. There are various results on the computation of $\bar{\partial} s(\bar x)$ in the literature; see \cite{DempeFoundations} and references therein. 

It is important to highlight the fact condition \eqref{QCDUAL} automatically holds if the mapping $\Lambda$ is Lipschitz-like holds at all the relevant points; cf. details in Section \ref{On the subdifferential}. Next, we provide an example showing that estimates in Theorem \ref{foulam} are still possible in some cases where condition \eqref{QCDUAL} fails.
\begin{example}\label{ExampleRalph}Consider the optimal value function
\begin{equation}\label{RalphDempeExample}
\varphi(x) = \underset{y}\min\left\{(y-1)^2|\; -x+y\leq 0, \; x+y\leq 0\right\},
\end{equation}
from \cite{RalphDempe1995}. We can easily check that the optimal solution map $S$ reduces to
$$
s(x)= \left\{\begin{array}{rll}
                          x & \mbox{ if } & x\leq 0, \\
                          -x & \mbox{ if } & x\geq 0. \\
                        \end{array}
             \right.
$$
The compactness of $\text{gph}\,K$, required in Theorem \ref{subdifferentialEstimate}(ii), does not hold. But \eqref{Sub-varphi} remain valid thanks to the continuity of the solution function $s$ and the fact the MFCQ holds at any feasible point. Problem \eqref{RalphDempeExample} is obviously convex in $y$ and calculations show that
$$
\Lambda(x, y)= \left\{\begin{array}{lll}
                          \left\{(0,\; 2(1-x))\right\} & \mbox{ if } & x < 0, \; y=x,\\
                          \left\{(2(1+x),\; 0)\right\} & \mbox{ if } & x > 0, \; y=-x,\\
                          \left\{(u_1, u_2)|\; u_1+u_2 = 2, \; u_1 \geq 0, \; u_2\geq 0\right\} & \mbox{ if } & x = 0, \; y=0.
                        \end{array}
             \right.
$$
We can easily show that condition \eqref{QCDUAL} fails. However, thanks to the fact that  $\Lambda$ is polyhedral, the detailed estimates (see \eqref{Yespapa}) for $\partial^2\varphi(\bar x|\underline{x})(\underline{x}^*)$ can still be obtained thanks to the calmness property; cf. discussions following Proposition \ref{teclair}.
\end{example}

\section{Generalized Hessian estimates in the absence of convexity}\label{Generalized Hessian estimates in the absence of convexity}
We assume here that the functions involved in \eqref{ParametricOptimization} are not necessarily convex.

\subsection{Single-valued optimal solution and multipliers maps}\label{Single-valued solution and multipliers maps}
We assume throughout this subsection that the optimal solution mapping $S$ \eqref{S(x)} and the Lagrange multipliers mapping $\Lambda$ \eqref{Lambdadefined} are all single-valued. 
Before we move to the general case, note that if $s$ is single-valued in \eqref{Danskin value function}, we can get the following result, where  the concave-convexity of $f$ or the qualification condition \eqref{pouolano} are not necessary.
\begin{thm}{\em
Consider the optimal value function $\varphi$ \eqref{Danskin value function} and let the corresponding optimal solution map $S$ \eqref{S(x)} be single-valued (i.e., $S:=s$) and Lipschitz continuous around $\bar x$, where $s(\bar x)=\bar y$. Then,
\begin{equation}\label{SingleValuedCase1}
    \partial^2 \varphi(\bar x)(\underline{x}^*) \subseteq \nabla^2_{xx}f(\bar x, \bar y) \underline{x}^* + \partial \left\langle\nabla^2_{xy}f(\bar x, \bar y) \underline{x}^*, s\right\rangle(\bar x).
\end{equation}
}\end{thm}
\begin{proof} Consider the function $\psi: \mathbb{R}^n \rightarrow \mathbb{R}^{n+m}$ defined  by $\psi(x):= (x^\top, s(x)^\top)^\top$. Then $\partial\varphi(x):= \nabla_x f \circ \psi(x)$ and we can check that
$$
\begin{array}{rll}
 \partial^2 \varphi(\bar x)(x^*) & \subseteq & \partial \left\langle \left[\begin{array}{c}
                                                                        \nabla^2_{xx}f(\bar x, \bar y)x^*\\
                                                                         \nabla^2_{xy}f(\bar x, \bar y)x^*
                                                                      \end{array}
  \right], \; \psi\right\rangle(\bar x)\\[2ex]
  & \subseteq & \nabla^2_{xx}f(\bar x, \bar y)x^* + \partial \left\langle\nabla^2_{xy}f(\bar x, \bar y) x^*, s\right\rangle(\bar x),
\end{array}
$$
while respectively using the chain and product rules in Theorems \ref{chain rule} and \ref{pp1}.
\end{proof}
Obviously, if $s$ is single-valued and continuously differentiable at $\bar x$,  we get the equality in \eqref{man-no-lap2}. Also, it can be useful to observe that  \eqref{man-no-lap1} provides an upper bound for the generalized Hessian of $\varphi$ which is looser than the one in \eqref{SingleValuedCase1}; cf. \eqref{Clarke Subifferential}.
Next, we show that under additional assumptions, inclusion \eqref{nemama} is valid in the case where the constraint function $g$  effectively depends on both $x$ and $y$.
\begin{thm}\label{inique}{\em Consider $\varphi$ \eqref{OptValueFn} and suppose that $\mbox{gph }K$ is compact and the MFCQ holds at $(\bar x, y)$, for all $y\in S(\bar x)$. Further assume that the mappings $S$ and $\Lambda$ are single-valued (i.e., $S:=s$ and $\Lambda:=\lambda$) and Lipschitz continuous around $\bar x$ and $(\bar x, \bar y)$, respectively,  with $\bar y= s(\bar x)$ and $\bar u=\lambda(\bar x, \bar y)$. Then, we have inclusion \eqref{nemama}.
}\end{thm}
\begin{proof} It is clear that with  $\mbox{gph }K$  compact and the MFCQ satisfied at $(\bar x, y)$, for all $y\in S(\bar x)$, we have from \cite{GauvinDubeau1982} that
\begin{equation}\label{subphi-KM}
   \bar{\partial} \varphi(x) \subseteq \mbox{ co }\underset{y\in S(x)}\bigcup \; \underset{u\in \Lambda (x,y)}\bigcup\; \left\{\nabla_x f(x,y) + \nabla_x g(x,y)^\top u \right\}
\end{equation}
holds near $\bar x$. $S$ and $\Lambda$ being both single-valued   around $\bar x$ and $(\bar x, \bar y)$, respectively, it follows from inclusion \eqref{subphi-KM} that we have
$
\partial \varphi(x)  =  \nabla_xL \circ \phi \circ \psi(x)
$
with $\psi$ and $\phi$ respectively defined by $\phi(a, b):=[a^\top, b^\top, \lambda(a,b)^\top]^\top$ and $\psi(x):=[x^\top, s(x)^\top]^\top$. Hence,
\begin{align}
\partial^2 \varphi(\bar x| x)(x^*) = & \;\, \partial \langle x^*, \nabla_xL \circ \phi \circ \psi\rangle (\bar x) \label{nema1}\\
                                                                        \subseteq  &   \;\, \partial \langle \nabla(\nabla_xL ) (\phi \circ \psi(\bar x))^\top x^*,  \; \phi \circ \psi \rangle   (\bar x) \label{nema2}\\
                                                                        \subseteq  &  \;\,  \underset{y^* \in \partial \langle \nabla(\nabla_xL ) (\phi \circ \psi(\bar x))^\top x^*,  \;\phi \rangle   (\psi(\bar x))}\bigcup \partial \langle y^*,  \psi \rangle   (\bar x) \label{nema3}
\end{align}
since the function $\nabla_xL$ is differentiable and $\phi$ and $\psi$ are Lipschitz continuous around $\psi(\bar x)$ and $\bar x$, respectively. One can easily observe that \eqref{nema2}--\eqref{nema3} result from the chain rule in Theorem \ref{chain rule}. It now remains to evaluate the subdifferentials involved in \eqref{nema3}. In fact, we can easily check that the following equalities hold:
\begin{align}
\nabla(\nabla_xL ) (\phi \circ \psi(\bar x))^\top x^* =& \;\, \left[\begin{array}{c}
                                                                  \nabla^2_{xx}L(\bar x, s(\bar x), \lambda(\bar x, s(\bar x)))x^* \\
                                                                  \nabla^2_{xy}L(\bar x, s(\bar x), \lambda(\bar x, s(\bar x)))x^* \\
                                                                  \nabla_x g(\bar x, s(\bar x))x^*
                                                                \end{array}
 \right], \label{nema4}\\
 \partial \langle \zeta^*, \phi\rangle(a,b) = & \;\, \left[\begin{array}{c}
                                                                 \zeta^*_x
                                                                  \zeta^*_y
                                                               \end{array}
  \right] + \partial \langle \zeta^*_u, \lambda\rangle(a,b), \label{nema5}\\
  \partial \langle \zeta^*, \psi\rangle(\bar x) = &  \;\, \zeta^*_x + \partial \langle \zeta^*_y, s \rangle(\bar x). \label{nema6}
\end{align}
It then follows from \eqref{nema4} and \eqref{nema5} that we have
$$
\begin{array}{l}
  y^* \in \partial \left\langle \nabla(\nabla_xL ) (\phi \circ \psi(\bar x))^\top x^*,  \;\phi \right\rangle   (\psi(\bar x))\\
   \qquad \qquad \qquad\qquad \qquad \qquad \Longleftrightarrow \left\{\begin{array}{l}
              y^*=     \left[\begin{array}{c}
                                                                  \nabla^2_{xx}L(\bar x, s(\bar x), \lambda(\bar x, s(\bar x)))x^* \\
                                                                  \nabla^2_{xy}L(\bar x, s(\bar x), \lambda(\bar x, s(\bar x)))x^*
                                                                \end{array}
 \right] + \zeta^*\\
  \mbox{with } \zeta^*\in \partial \langle \nabla_x g(\bar x, s(\bar x))x^*, \;\,\lambda \rangle (\bar x, \bar y).
                                                                                                                                                    \end{array}
 \right.
\end{array}
$$
Substituting this, together with \eqref{nema6}, in \eqref{nema3},  we arrive at inclusion \eqref{nemama}.
\end{proof}
This result is clearly a special case of Theorem \ref{foulam}, as both estimates will coincide if $\Lambda$ is locally single-valued and Lipschitz continuous. Note that the latter property ensures that qualification condition \eqref{QCDUAL} automatically holds.
If in addition to the assumptions made in Theorem \ref{inique}, we suppose that the functions  $s$ and $\lambda$ are differentiable at $\bar x$ and $(\bar x, \bar y)$, respectively, then we have equality \eqref{mamba}.  To ensure that the latter holds, recall that we have already discussed in Subsection \ref{Single-valued optimal solution map} how to get the function $s$ locally single-valued and Lipschitz continuous. If stronger assumptions are made, i.e., the SSOSC, LICQ, and  strict complementarity condition (i.e., $\theta=\emptyset$) for $(\bar x, \bar y)$, then the optimal solution set-valued mapping $S$ \eqref{S(x)} and the Lagrange multiplier mapping $u$ (as a function of just $x$) are locally unique and based on the implicit function theorem, their derivatives can be obtained from the system
\begin{equation}\label{ImplicitFnThm}
\begin{array}{c}
   \left[
\begin{array}{cccc}
  \nabla^2_{yy}L(\bar x, \bar y, \bar u) & \nabla_y g_1(\bar x, \bar y)^\top & \ldots & \nabla_y g_p(\bar x, \bar y)^\top \\
  \bar u_1 \nabla_y g_1(\bar x, \bar y)& g_1(\bar x, \bar y)& \ldots & O\\
  \vdots & \vdots & \ddots &  \\
  \bar u_p \nabla_y g_p(\bar x, \bar y) & O & \ldots & g_p(\bar x, \bar y)
\end{array}
\right]\left[\begin{array}{c}
               \nabla s(\bar x)\\
               \nabla u(\bar x)
             \end{array}
 \right]  = \left[
\begin{array}{c}
  \nabla^2_{xx}L(\bar x, \bar y, \bar u) \\
  \bar u_1 \nabla_x g_1(\bar x, \bar y)\\
  \vdots\\
  \bar u_p \nabla_x g_p(\bar x, \bar y)
\end{array}
\right]
\end{array}
\end{equation}
cf. \cite[Chapter 3]{FiaccoBook1983}. Under additional invertibility assumptions, complete expressions of $\nabla s(\bar x)$ can be written in terms of the problem data; see, e.g., \cite[Chapter 7]{ShimizuIshizukaBardBook1997}.

\subsection{Set-valued optimal solution map}\label{Set-valued optimal solution map}
We start here by considering the case $\varphi$ \eqref{Danskin value function}, but while further assuming that the involved functions are not necessarily convex. Unlike in the previous subsection, we let $S$ \eqref{S(x)} be set-valued throughout this subsection. We denote by
$$
\begin{array}{rl}
  \Gamma^{\circ} (\bar x, \underline{x}):= \left\{\left.(a,z)\in \mathbb{R}^{n+1}\times \prod^{n+1}_{s=1}\mathbb{R}^{m}\right|\right.& \; a\geq 0, \; \sum^{m+1}_{s=1}a_s=1, \\ & \left.\sum^{m+1}_{s=1}a_s\nabla_x f(\bar x, z^s)=\underline{x}, \; z\in \prod^{n+1}_{s=1}S(\bar x)\right\}
\end{array}
$$
and
$$
\Delta^{\circ} (\bar x, z^s):= \left\{y\in \mathbb{R}^m|\; y\in S(\bar x):\; \nabla_x f(\bar x, y)= \nabla_x f(\bar x, z^s)  \right\}.
$$
\begin{thm}\label{mekouna}{\em
Consider the optimal value function $\varphi$ \eqref{Danskin value function} and let $Y$ be a compact set. Furthermore, consider a point $(\bar x, \underline{x})$ such that $\underline{x}\in \partial \varphi(\bar x)$ and the implication
\begin{equation}\label{pouolano}
\begin{array}{c}
\left[ \sum^{m+1}_{s=1}v^s=0, \; v^s\in \underset{y\in \Delta^{\circ} (\bar x, z^s)}\bigcup D^*S\left(\bar x| y\right)(0), \; s=1, \ldots, n+1\right] \\[2ex] \qquad \qquad \qquad \qquad \qquad\qquad \qquad \qquad \qquad\qquad \qquad \Longrightarrow v^1=\ldots = v^{n+1}=0
\end{array}
\end{equation}
 is satisfied at all $(a, z)\in \Gamma^{\circ}(\bar x, \underline{x})$. Then, for all $\underline{x}^*\in \mathbb{R}^n$, it holds that
$$
 \begin{array}{l}
 \partial^2 \varphi(\bar x| \underline{x})(\underline{x}^*) \subseteq \\[1.5ex]
 \qquad \underset{(a, z)\in \Gamma^{\circ} (\bar x| \underline{x})}\bigcup \left\{\sum^{n+1}_{s=1} \left[ \underset{ y\in \Delta^{\circ} (\bar x, z^s)} \bigcup \left[a_s\nabla^2_{xx}f(\bar x, y)\underline{x}^* + D^*S(\bar x|y)\left(a_s\nabla^2_{xy}f(\bar x, y)\underline{x}^*\right)\right]\right]\right\}.
\end{array}
$$
}\end{thm}
\begin{proof}
Let us first recall that under the compactness of $Y$ and the continuously differentiability of $f$, we have from \eqref{copartialphi} that $\bar \partial \varphi(x)= \mbox{co } \nabla_x f \circ \Psi(x)$ with $\Psi(x):= \{x\}\times S(x)$. Next, recall that $\Psi$ is closed, following the discussion in the proof of Theorem \ref{partial2phi}. Furthermore, one can easily check that $\Psi$ is locally bounded around any point in $\mathbb{R}^n$.
Now, consider a sequence $\{(a^k, c^k)\}$ with $c^k\in \nabla_x f \circ \Psi(a^k)$ such that $a^k \rightarrow \bar a$ and $c^k \rightarrow \bar c$. Obviously, we can find a sequence $\{b^k\}$ with $b^k \in \Psi(a^k)$ such that $\nabla_x f (b^k) =c^k$ for all $k$. By the local boundedness of $\Psi$, $\{b^k\}$ admits  a convergent subsequence that we denote similarly, provided there is no confusion. By the closedness of $\Psi$, we have $b^k \rightarrow \bar b\in \Psi(\bar a)$, for some $\bar b \in \mathbb{R}^{n+m}$. Additionally, note that $\nabla_x f(\bar b) = \bar c$, as $f$ is assumed to be continuously differentiable throughout this paper. Thus $\bar c\in \nabla_xf\circ\Psi(\bar a)$; confirming that $\nabla_xf \circ \Psi$ is a closed set-valued mapping.
Next, we consider a sequence $x^k \rightarrow \bar x$ and any sequence $z^k\in \nabla_xf \circ \Psi (x^k)$. Then, we can find a sequence $\{y^k\}$ such that $y^k\in S(x^k)$ and $z^k= \nabla_xf(x^k, y^k)$. By definition of the counterpart of $S$ \eqref{S(x)} for \eqref{Danskin value function}, it follows that $y^k\in Y$ for all $k$. Hence, as $Y$ is compact, it follows from the well-known Bolzano-Weierstrass Theorem that  $\{y^k\}$ has a convergent subsequence, that we denote similarly, provided there is no confusion. Let $\bar y$ be the limit of this subsequence. Then, we have $\bar y\in S( \bar x)$, given that $S$ is closed as shown in Theorem \ref{partial2phi}. Subsequently, as  $f$ is continuously differentiable, the sequence $\{z^k\}$ converges to $\nabla_x f(\bar x, \bar y)$. This shows that $\nabla_xf \circ \Psi$ is locally bounded around $\bar x$.
It then follows from Proposition \ref{transmon} that if the counterpart of \eqref{QC-Cod} for the mapping in \eqref{copartialphi} holds at all $(a,b)\in \Gamma(\bar x, \underline{x})$, it holds that
\begin{equation}\label{gheje}
    \partial^2 \varphi(\bar x| \underline{x})(\underline{x}^*) \subseteq \underset{(a, b)\in \Gamma (\bar x| \underline{x})}\bigcup \left[\sum^{n+1}_{s=1} D^*\left(\nabla_xf\circ \Psi\right)\left(\bar x| b_s\right)\left(a_s \underline{x}^*\right)\right].
\end{equation}
Note that the mapping $\Gamma$ in \eqref{gheje} corresponds to the counterpart of \eqref{dfort} for the convex hull set-valued mapping in \eqref{copartialphi}. \eqref{gheje} leads to the estimate in the theorem, considering inclusions \eqref{Hessian phi 1} and \eqref{Cod Psi3} and the fact that $(a,b)\in \Gamma(\bar x, \underline{x})$ if and only if
$$
\exists z \, \mbox{ s.t. }\, (a,z)\in \Gamma^{\circ}(\bar x, \underline{x}) \mbox{ with } b=\left[\begin{array}{c}
\nabla_x f(\bar x, z^1)^\top, \ldots, \nabla_x f(\bar x, z^{n+1})^\top                                                                                                                                                          \end{array}
\right]^\top.
$$
Finally, observe that \eqref{pouolano} is a sufficient condition for  the  counterpart of qualification condition \eqref{QC-Cod} for \eqref{copartialphi} to hold.
\end{proof}
Observe that based on the coderivative criterion \eqref{coderivative criterion}, the qualification condition \eqref{pouolano} is automatically satisfied if  $S$ is Lipschitz-like around $(\bar x, y)$ for all $y\in \Delta(\bar x, z^s)$, $(a, z)\in \Gamma^{\circ}(\bar x, \underline{x})$. Also, it is clear that the estimate of generalized Hessian of $\varphi$ obtained in Theorem  \ref{partial2phi} is much tighter than the one derived in Theorem \ref{mekouna}.

Similarly to Theorem \ref{mekouna}, we now consider $\varphi$ in the general case Consider \eqref{OptValueFn}, but without the convexity assumption. Then we have the following result, where
$$
\begin{array}{l}
  \Gamma^{\lambda} (\bar x, \underline{x}):= \Big\{\left.(a,z, w)\in \mathbb{R}^{n+1}\times \prod^{n+1}_{s=1}\mathbb{R}^{m} \times \prod^{n+1}_{s=1}\mathbb{R}^{p}\right|\; a\geq 0, \; \sum^{n+1}_{s=1}a_s=1,\\
  \qquad \qquad \qquad \qquad  \sum^{n+1}_{s=1}a_s\nabla_x L(\bar x, z^s, w^s)=\underline{x}, \; z\in \prod^{n+1}_{s=1}S(\bar x), \; w\in \prod^{n+1}_{s=1}\left\{\lambda(\bar x, z^s)\right\}\Big\}
\end{array}
$$
and
$$
\Delta^\lambda (\bar x, z^s, w^s):= \Big\{(y,u)|\; y\in S(\bar x), \;\;u=\lambda(\bar x, y), \;\nabla_x L(\bar x, y, u)= \nabla_x L(\bar x, z^s, w^s) \Big\}.
$$
\begin{thm}\label{fonii}{\em Consider \eqref{OptValueFn} and suppose that $\mbox{gph }K$ is compact and the LICQ holds at all $(x, y)\in \mbox{gph }S$. Further assume that $\Lambda$ is single-valued  and Lipschitz continuous around $(\bar x, y)$ with $u =\lambda(\bar x, y)$ for all $y\in S(\bar x)$ such that $ \nabla_x L(\bar x, y, u)=\underline{x}$. Furthermore, let $S$ \eqref{S(x)} be closed and  locally bounded around $\bar x$, and the qualification condition
\begin{equation}\label{condre12}
\begin{array}{c}
\left[ \sum^{m+1}_{s=1}v_s=0, \; v_s\in \underset{y\in \Delta (\bar x, z^s, w^s)}\bigcup D^*S\left(\bar x| y\right)(0), \; s=1, \ldots, n+1\right] \\ \qquad \qquad \qquad \qquad\qquad \qquad \qquad \qquad\qquad \Longrightarrow v_1=\ldots = v_{n+1}=0
\end{array}
\end{equation}
 holds at all $(a, z, w)\in \Gamma^{\lambda}(\bar x, \underline{x})$, where $\underline{x}\in \varphi(\bar x)$. Then, for all $\underline{x}^*\in \mathbb{R}^n$, it holds that
\begin{align}
\varphi^2(\bar x|\underline{x})(\underline{x}^*) \subseteq  &\underset{(a,z,w)\in \Gamma^{\lambda}(\bar x, \underline{x})} \bigcup \Bigg\{\sum^{n+1}_{s=1} \Bigg[ \underset{(y, u)\in \Delta^{\lambda}(\bar x, z^s, w^s)}\bigcup \;\;\underset{(\zeta^*_x, \zeta^*_y)\in \;a_s\partial \langle \nabla_x g(\bar x, y)x^*, \;\,\lambda \rangle (\bar x, y)}\bigcup \nonumber\\[1ex]
  &   \left[a_s\nabla^2_{xx}L(\bar x, y, u)\underline{x}^* + \zeta^*_x +   D^*S(\bar x|y)\left(\zeta^*_y + a_s\nabla^2_{xy}L(\bar x, y, u)\underline{x}^*\right)\right]\Bigg]\Bigg\}.\nonumber
\end{align}
}\end{thm}
\begin{proof} Based on the compactness of $\mbox{gph }K$, the closedness of $S$ and the Bolzano-Weierstrass Theorem we can easily show that  $\nabla_x L \circ \phi \circ \Psi$ is  closed and locally bounded around $\bar x$, thanks to the locally Lipschitz continuity of $\phi$. The rest of the proof follows the steps of that of Theorem \ref{mekouna}.
  \end{proof}
Similarly to the previous result, the qualification condition \eqref{condre12} is automatically satisfied if  $S$ is Lipschitz-like around $(\bar x, y)$ for all $y\in \Delta(\bar x, z^s)$, $(a, z, w)\in \Gamma^{\lambda}(\bar x, \underline{x})$. The only remaining thing to clarify here is how to estimate the coderivative of $S$, as well as ensuring that $S$ is Lipschitz-like, in the absence of convexity. To proceed, we simply restate the following result from \cite[Theorem 5.9]{DempeMordukhovichZemkohoTwo-level}.
\begin{thm}\label{Sensitivity of S opt val case}{\em
Let the mapping $S$ \eqref{S(x)} be inner semicontinuous at $(\bar{x},\bar{y})\in{\rm gph}\,S$, while the set-valued map
$
\Xi(v):=\{(x,y)|\; g(x,y)\leq 0, \; f(x,y)-\varphi(x)\leq v\}
$
is calm at $(0, \bar{x}, \bar{y})$. Furthermore, if the MFCQ holds at $(\bar{x},\bar{y})$, then for all $y^*\in\mathbb{R}^m$,
\begin{equation}\label{cod-S}
D^*S(\bar{x}|\bar{y})(y^*) \subseteq   \underset{(\beta, \lambda)\in \bar{\Lambda} (\bar{x}, \bar{y}, y^*)}\bigcup \left\{\lambda \left(\nabla_x f(\bar{x}, \bar{y}) -  \bar{\partial} \varphi(\bar{x})\right) + \nabla_x g(\bar{x}, \bar{y})^\top\beta\right\}
\end{equation}
with $\bar{\Lambda} (\bar{x}, \bar{y}, y^*)$ collecting all $(\beta, \lambda)$ verifying $\lambda\geq 0$, $\beta \geq 0$, $\beta^\top g(\bar{x}, \bar{y})=0$ such that
$$
y^* + \lambda \nabla_y f(\bar{x}, \bar{y}) + \nabla_y g(\bar{x}, \bar{y})^\top \beta =0.
$$
Furthermore $S$ is Lipschitz-like around $(\bar x, \bar y)$ if in addition, it holds that
\begin{equation*}\label{A52}
\big[(\beta, \lambda)\in \bar{\Lambda} (\bar{x}, \bar{y}, 0), \,\,u\in \partial(-\varphi)(\bar{x})\big]\Longrightarrow \lambda u + \lambda
\nabla_x f(\bar{x}, \bar{y}) + \nabla_x g(\bar{x}, \bar{y})^\top \beta =0.
\end{equation*}
}\end{thm}
Recall that to get $\bar{\partial} \varphi(\bar{x})$ in \eqref{cod-S}, we can use \eqref{copartialphi} and \eqref{subphi-KM-II} for Theorems \ref{mekouna} and \ref{fonii}, respectively. For more details on this result, as well as examples where the assumptions hold, can be found in \cite{DempeMordukhovichZemkohoTwo-level}.

\section{Final discussion and future work}\label{Final discussion and future work}
Our focus in this paper was to provide upper estimates for the generalized Hessian of $\varphi$ under different scenarios. In some cases, these upper estimates may be substantially larger than the generalized Hessians being estimated. This is, for example, the case for the problem in Example \ref{ExampleRalph}, which leads to
\[
\begin{array}{l}
  \mbox{gph}\,\left(\partial \varphi\right) = A \cup B \cup C \\
  \mbox{with }\left\{\begin{array}{l}
                        A:=\left\{(x, y)\in \mathbb{R}^2\left|~x\leq 0, \;\; y=2(x-1)\right.\right\}, \\
                        B:=\left\{(x, y)\in \mathbb{R}^2\left|~x\geq 0, \;\; y=2(x+1)\right.\right\},\\
                        C:= \left\{0\right\}\times \left[-2, \;\; 2 \right].
                     \end{array}
   \right.
\end{array}
\]
Then based on partitions of the sets $A$, $B$, and $C$, we introduce the following sets:
\[
\begin{array}{l}
  \Delta_1 := \left\{\left(x, y\right)\in \mathbb{R}^2\left|~x<0, \;\;\; y=2(x-1)\right.\right\},\\
  \Delta_2 := \left\{\left(x, y\right)\in \mathbb{R}^2\left|~x>0, \;\;\; y=2(x+1)\right.\right\},\\
  \Delta_3 := \left\{\left(x, y\right)\in \mathbb{R}^2\left|~x=0, \;\;\; y=-2\right.\right\},\\
  \Delta_4 := \left\{\left(x, y\right)\in \mathbb{R}^2\left|~x=0, \;\;\; y=2\right.\right\},\\
  \Delta_5 := \left\{\left(x, y\right)\in \mathbb{R}^2\left|~x=0, \;\;\; y\in \left(-2, \;\,2\right)\right.\right\}.
\end{array}
\]
By basic calculations, see, e.g., \cite{{DempeZemkohoKKT-SIAM-paper2},Zemkoho2016Solving}, we can check that the Fr\'{e}chet normal cone to the graph of $\partial\varphi$ can be obtained as
\[
\hat{N}_{\mbox{gph}\left(\partial \varphi\right)}\left(\bar x, \bar y \right)
=\left\{
\begin{array}{lll}
\Omega_1 := \left\{\left(x, y\right)\in \mathbb{R}^2\left|x+2y =0\right.\right\}& \mbox{ if } &  (\bar x, \bar y)\in \Delta_1 \cup \Delta_2,\\
\Omega_2 := \left\{\left(x, y\right)\in \mathbb{R}^2\left|x+2y \geq 0, \;\, y\leq 0\right.\right\}& \mbox{ if } & (\bar x, \bar y)\in \Delta_3,\\
\Omega_3 := \left\{\left(x, y\right)\in \mathbb{R}^2\left|x+2y \leq 0, \;\, y\geq 0\right.\right\}& \mbox{ if } & (\bar x, \bar y)\in \Delta_4,\\
\Omega_4 := \mathbb{R}\times \{0\}& \mbox{ if } &  (\bar x, \bar y)\in \Delta_5.
\end{array}
 \right.
\]
It follows from this equality and considering the definition of the basic normal cone given in \eqref{basic normal cone}, we have
\[
N_{\mbox{gph}\left(\partial \varphi\right)}\left(\bar x, \bar y \right)
=\left\{
\begin{array}{lll}
\Omega_1 & \mbox{ if }               &  (\bar x, \bar y)\in \Delta_1 \cup \Delta_2,\\
\Omega_2 \cup \Omega_1 \cup \Omega_4 & \mbox{ if } & (\bar x, \bar y)\in \Delta_3,\\
\Omega_3 \cup \Omega_1 \cup \Omega_4 & \mbox{ if } & (\bar x, \bar y)\in \Delta_4,\\
\Omega_4  & \mbox{ if } &  (\bar x, \bar y)\in \Delta_5.
\end{array}
 \right.
\]
Hence, from the definition of the concept of coderivative, it holds that
\begin{equation}\label{SecondOrderPhi}
\partial^2 \varphi(\bar x|\bar y)(y^*)
=\left\{
\begin{array}{lll}
\{2y^*\} & \mbox{ if }               &  (\bar x, \bar y)\in \Delta_1 \cup \Delta_2, \;\; y^*\in \mathbb{R},\\
\mathbb{R} & \mbox{ if } & (\bar x, \bar y)\in \Delta_3, \;\; y^*=0,\\
\left[2y^*, \; \infty\right) & \mbox{ if } & (\bar x, \bar y)\in \Delta_3, \;\; y^* > 0,\\
\{2y^*\} & \mbox{ if } & (\bar x, \bar y)\in \Delta_3, \;\; y^*<0,\\
\mathbb{R} & \mbox{ if } & (\bar x, \bar y)\in \Delta_4, \;\; y^*=0,\\
\left(-\infty, \; 2y^*\right] & \mbox{ if } & (\bar x, \bar y)\in \Delta_4, \;\; y^* < 0,\\
\{2y^*\} & \mbox{ if } & (\bar x, \bar y)\in \Delta_4, \;\; y^*>0,\\
\mathbb{R}  & \mbox{ if } &  (\bar x, \bar y)\in \Delta_5, \;\; y^*=0,\\
\emptyset  & \mbox{ if } &  (\bar x, \bar y)\in \Delta_5, \;\; y^*\neq 0.
\end{array}
 \right.
\end{equation}

We can easily conclude from the formula \eqref{SecondOrderPhi} that for $\bar x<0$ and $\underline{x} = \bar x$, we have
\[
\partial^2 \varphi(\bar x|\underline{x})(\underline{x}^*)=\emptyset \;\; \mbox{ for any }\;\; \underline{x}^*\in \mathbb{R}.
\]
On the other hand, to apply the upper estimate in \eqref{Yespapa}, we can start by noting that we have $\Lambda(\bar x, \underline{x})=\left\{(0, \;\; 2(1-\bar x))\right\}$ for $\underline{x} = \bar x < 0$. Hence, it holds that
\[
\begin{array}{lll}
D^*\Lambda(\bar x, \underline{x}|  u)\left(\nabla_x g(\bar x, \underline{x})\underline{x}^*\right) & = & D^*\Lambda\left(\bar x, \underline{x}\left|\;0,\; 2(1-\bar x)\right.\right)\left(-\underline{x}^*, \underline{x}^*\right),\\[1ex]
  & = & \left\{(x^*, y^*)\in \mathbb{R}^2|\;\; x^* + y^* + 2\underline{x}^*\geq 0\right\},
\end{array}
\]
for $\underline{x}^*\in \mathbb{R}$, where we denote by $u :=(0, \;\; 2(1-\bar x))$ and $\nabla_x g(\bar x, \underline{x})\underline{x}^* := \left(-\underline{x}^*, \;\underline{x}^*\right)^\top$. Therefore, based on \eqref{Yespapa}, it follows that
\[
\partial^2 \varphi(\bar x|\underline{x})(\underline{x}^*)\subset \left\{\zeta^*_x + \zeta^*_y\left|\;\, \left(\zeta^*_x, \zeta^*_y\right)\in \mathbb{R}^2, \;\;\, \zeta^*_x + \zeta^*_y + 2\underline{x}^* \geq 0\right.\right\} \neq \emptyset.
\]
This confirms that in the case of this example, the upper estimate of $\partial^2 \varphi(\bar x|\underline{x})(\underline{x}^*)$ obtained from \eqref{Yespapa} is larger than the actual generalized Hessian of $\varphi$.

Another important observation from \eqref{SecondOrderPhi} is that for a specific structure of the functions $f$ and $g$ that define $\varphi$, exact formulas for $\partial^2 \varphi(\bar x|\underline{x})(\underline{x}^*)$ can be obtained. The generalization of such a calculation to broader classes of quadratic or linear versions of $f$ and $g$ will be carefully studied in a separate work. 

To conclude this section, we provide a flavor of how the results in this paper can be used in practice, by illustrating how they can be used in an approximation scheme to solve a bilevel optimization problem defined by
\begin{equation}\label{P}
\begin{array}{l}
   \underset{x,y}\min~F(x,y) \; \mbox{ s.t. }\;  G(x,y)\leq 0,\; y\in S(x):= \arg\underset{y}\min~\{f(x,y)\left|\; g(x,y)\leq 0\right.\},
\end{array}
\end{equation}
where $F,\; f :\mathbb{R}^n\times \mathbb{R}^m \rightarrow \mathbb{R}$, $G :\mathbb{R}^n\times \mathbb{R}^m \rightarrow \mathbb{R}^p$,  and $ g: \mathbb{R}^m \rightarrow \mathbb{R}^q$ are all continuously differentiable functions.  
Under the well-known partial calmness concept (see, e.g., \cite{mehlitz2021note, ye1995optimality}), a point that is locally optimal for problem \eqref{P} is also a local optimal solution for the partially penalized problem
\begin{equation}\label{Penalized-LLVF}
    \underset{x,y}\min~F(x,y) + \eta (f(x,y) -\varphi(x)) \;\mbox{ s.t. } \; G(x,y)\leq 0, \; g(x,y)\leq 0,
\end{equation}
for some penalization parameter $\eta  >0$.
Hence, for a given value of $\eta$, a point $(x, y)$ will be said to be Karush-Kuhn-Tucker (KKT)--stationary for problem \eqref{Penalized-LLVF} if there exists a Lagrange multiplier vector $(u, v)\in\mathbb{R}^{p+q}$  such that the condition
\begin{equation}\label{EquationToSolve}
0\in \Phi^\eta (\zeta):=\left[\begin{array}{c}
                                   H^\eta (\zeta)\\
                                   \phi_G(\zeta)\\
                                   \phi_g(\zeta)
                                 \end{array}
\right] + \bar\partial \varphi(x) \times \{0\}
\end{equation}
holds, with $\zeta:=\left(x, y, u, v \right)$ and the functions $H^\eta$, $ \phi_G$, and $ \phi_g$ respectively defined by
\[
\begin{array}{rll}
 H^\eta (\zeta) &:= & -\frac{1}{\eta}\left[\nabla F(x,y) + \nabla G(x, y)^{\top}u + \nabla g(x, y)^{\top}v\right] - \nabla f(x,y),\\
 \phi_G(\zeta)& := & \left(\min\{-G_i(x,y), \;\; u_i\}\right)_{i=1, \ldots, p},\\
 \phi_g(\zeta)& := & \left(\min\{-g_j(x,y), \;\; v_j\}\right)_{j=1, \ldots, q}.
\end{array}
\]
Considering recent developments in the construction of Newton-type methods for nonsmooth inclusions, using the Mordukhovich coderivative (see, e.g.,  \cite{khanh2020generalized, mordukhovich2020generalized, mordukhovich2020globally}), we can solve the following equation in $d$ at iteration $k$:
\begin{equation}\label{NewtonStep}
\Psi^{\eta, k} (d) = - \bar{\Phi}^\eta \left(\zeta^k, \underline{x}^k\right),
\end{equation}
where
\[
\begin{array}{l}
  \Psi^{\eta, k} (d) := \left[\begin{array}{c}
                                   \nabla  H^\eta (\zeta^k)\\
                                   a(\zeta^k)\\
                                   b(\zeta^k)
                                 \end{array}
   \right]d + \left[\begin{array}{c}
                                  c(\zeta^k)(d_x)\\
                                  0\\
                                   0\\
                                   0
                                 \end{array}
   \right] \;\; \mbox{ and } \;\; \bar{\Phi}^\eta  (\zeta^k, \underline{x}^k):=\left[\begin{array}{c}
                                   H^\eta (\zeta^k)\\
                                   \phi_G(\zeta^k)\\
                                   \phi_g(\zeta^k)
                                 \end{array}
\right] + \left[\begin{array}{c}
                                   \underline{x}^k\\
                                   0\\
                                   0\\
                                   0
                                 \end{array}
\right],\\[4ex]
   a(\zeta^k)\in \bar\partial \phi_G(\zeta^k), \;\;\; b(\zeta^k)\in \bar\partial \phi_g(\zeta^k),\;\;\;  c(\zeta^k)(d_x)\in \partial^2 \varphi\left(x^k|\underline{x}^k\right)(d_x),\\[2ex]
   \underline{x}^k \in \bar\partial \varphi(x^k) \;\; \mbox{ such that }\;\; H^\eta_x(\zeta^k) + \underline{x}^k=0 \;\; \mbox{ with } \;\; H^\eta_x(\zeta^k)\;\; \mbox{ defined by}\\[1ex]
   H^\eta_x(\zeta^k):=-\frac{1}{\eta}\left[\nabla_x F(x^k,y^k) + \nabla_x G(x^k, y^k)^{\top}u^k + \nabla_x g(x^k, y^k)^{\top}v^k\right] - \nabla_x f(x^k, y^k).
\end{array}
\]
After computing $d:=d^k$ from the generalized Newton step from \eqref{NewtonStep}, a new iterate, as usual, can be obtained as
\begin{equation}\label{NewIter}
\zeta^{k+1}= \zeta^k + d^k.
\end{equation}
As it is well-known how to calculate elements from $\bar\partial \phi_G(\zeta^k)$ and $\bar\partial \phi_g(\zeta^k)$, the first main question here is how to calculate an element from $\partial^2 \varphi\left(x^k|\underline{x}^k\right)(d_x)$. In the case where it is not possible to have an exact formula for this generalized Hessian as in the example above, see \eqref{SecondOrderPhi}, elements from the upper estimates obtained in the fourth and fifth sections of this paper could be used for a heuristic scheme to solve inclusion \eqref{EquationToSolve}. A separate work is ongoing to study the theoretical properties of the generalized Newton scheme \eqref{NewtonStep}--\eqref{NewIter} in the case where an exact formula for $\partial^2 \varphi\left(x^k|\underline{x}^k\right)(d_x)$ can be obtained, and also evaluate the practical efficiency in the case where only upper bounds of this generalized Hessian can be estimated. \\[1ex]

\textbf{Acknowledgements.} The author would like to thank the anonymous referees for their constructive feedback that led to improvements in the paper.
\bibliographystyle{plain}

\end{document}